\documentclass[12pt,reqno]{amsart}

\usepackage{amsmath,amsfonts,amsthm,amssymb,color}
\usepackage{pdfsync}
\usepackage[latin1]{inputenc}
\usepackage{graphicx}
\usepackage{pstricks}
\usepackage{lmodern}

\newcommand{\der}{\delta}

\newcommand{\pt}[1]{\overset{#1}{\frown}_q}
\newcommand{\contra}[1]{\stackrel{#1}{\frown}}

\newcommand{\ca}{{\mathcal A}}

\newcommand{\al}{\alpha}

\newcommand{\si}{\sigma}

\newcommand{\vp}{\varphi}

\newcommand{\N}{{\mathbb N}}

\newcommand{\R}{{\mathbb R}}
\newcommand{\Z}{{\mathbb Z}}

\newcommand{\bean}{\begin{eqnarray*}}
\newcommand{\eean}{\end{eqnarray*}}
\newcommand{\ben}{\begin{enumerate}}
\newcommand{\een}{\end{enumerate}}
\newcommand{\beq}{\begin{equation}}
\newcommand{\eeq}{\end{equation}}

\newtheorem{theorem}{Theorem}[section]

\newtheorem{definition}[theorem]{Definition}

\newtheorem{lemma}[theorem]{Lemma}

\newtheorem{proposition}[theorem]{Proposition}

\theoremstyle{remark}
\newtheorem{remark}[theorem]{Remark}

\title{Fourth Moment Theorem and $q$-Brownian Chaos}

\begin{document}

\begin{center}
{\large\textbf{
Fourth Moment Theorem and $q$-Brownian Chaos
}}\\~\\
Aur\'elien Deya\footnote{Institut \'Elie Cartan, Universit\' e
de Lorraine, BP 70239, 54506 Vandoeuvre-l\`es-Nancy, France. Email: {\tt Aurelien.Deya@iecn.u-nancy.fr}},
Salim Noreddine\footnote{Laboratoire de Probabilit\'es et Mod\`eles Al\'eatoires, Universit\'e Paris 6, Bo\^ite courrier 188, 4 Place Jussieu, 75252 Paris Cedex 5, France. Email:
{\tt salim.noreddine@polytechnique.org}}
and Ivan Nourdin\footnote{Institut \'Elie Cartan, Universit\' e
de Lorraine, BP 70239, 54506 Vandoeuvre-l\`es-Nancy, France. Email: {\tt inourdin@gmail.com}.
Supported in part by the two following (french) ANR grants: `Exploration des Chemins Rugueux'
[ANR-09-BLAN-0114] and `Malliavin, Stein and Stochastic Equations with Irregular Coefficients'
[ANR-10-BLAN-0121].}
\end{center}

\bigskip

{\small \noindent {\bf Abstract:} In 2005, Nualart and Peccati \cite{NP} showed the so-called {\it Fourth Moment Theorem} asserting that, for a sequence of normalized multiple Wiener-It\^o integrals to converge to the standard Gaussian law, it is necessary and sufficient
that its fourth moment tends to 3. A few years later, Kemp {\it et al.} \cite{knps} extended this theorem to a sequence of normalized multiple Wigner integrals, in the context of the free Brownian motion.
The $q$-Brownian motion, $q\in(-1,1]$,
introduced by the physicists Frisch and Bourret \cite{FB} in 1970 and mathematically studied by Bo$\dot{\rm z}$ejko and Speicher \cite{BSp1}
in 1991, interpolates between the classical Brownian motion ($q=1$) and the free Brownian motion $(q=0)$, and is one of the nicest examples
of non-commutative processes. The question we shall solve in this paper is the following: what does the Fourth Moment Theorem become
when dealing with a $q$-Brownian motion?
\bigskip

\bigskip


\section{Introduction and main results}

The $q$-Brownian motion was introduced in 1970 by the physicists Frisch and Bourret \cite{FB} as an intermediate model between two standard theoretical axiomatics (see also \cite{Gre} for another physical interpretation). From a probabilistic point of view, it may be seen as a smooth and natural interpolation between two of the most fundamental processes in probability theory: on the one hand, the \emph{classical Brownian motion} $(W_t)_{t\geq 0}$ defined on a classical probability space $(\Omega,\mathcal{F},P)$; on the other hand, the \emph{free Brownian motion} $(S_t)_{t\geq 0}$ at the core of Voiculescu's free probability theory and closely related to the study of large random matrices (see \cite{voiculescu}).

\smallskip

The mathematical construction of the $q$-Brownian motion is due to Bo$\dot{\rm z}$ejko and Speicher \cite{BSp1}, and it heavily relies on the theory of \emph{non-commutative probability spaces}. Thus,
before describing our results and for the sake of clarity, let us first introduce some of the central concepts of this theory (see \cite{nicaspeicher} for a systematic presentation).

\smallskip

A {\em $W^\ast$-probability space} (or a non-commutative probability space) is a von Neumann algebra $\mathcal{A}$ (that is,
an algebra of bounded operators on a real
separable Hilbert space,
closed under adjoint and convergence in the weak operator topology) equipped with a {\em trace} $\vp$,
that is, a unital linear functional (meaning preserving the identity) which is weakly
continuous, positive (meaning $\vp(X)\ge 0$
whenever $X$ is a non-negative element of $\mathcal{A}$; i.e.\ whenever $X=YY^\ast$
for some $Y\in\mathcal{A}$), faithful (meaning that if
$\vp(YY^\ast)=0$ then $Y=0$), and tracial (meaning that $\vp(XY)=\vp(YX)$ for all
$X,Y\in\mathcal{A}$, even though in general $XY\ne YX$).

In a $W^\ast$-probability space $(\ca,\vp)$, we refer to the self-adjoint elements of the
algebra as
{\em random variables}.  Any
random variable $X$ has
a {\em law}: this is the unique compactly supported
probability measure $\mu$ on $\R$
with the same moments as $X$; in other words, $\mu$ is such that
\begin{equation}\label{mu}
\int_{\R} Q(x) d\mu(x) = \vp(Q(X)),
\end{equation}
for any real polynomial $Q$. Thus, and as in the classical probability theory, the focus is more on the \emph{laws} (which, in this context, is equivalent to the sequence of \emph{moments}) of the random variables than on the underlying space $(\ca,\vp)$ itself. For instance, we say that a sequence $\{X_k\}_{k\geq 1}$ of random variables such that $X_k \in (\ca_k,\vp_k)$ converges to $X\in (\ca,\vp)$ if, for every positive integer $r$, one has $\vp_k(X_k^r ) \to \vp ( X^r)$ as $k\to\infty$. In the same way, we consider here that any family $\{X^i\}_{i\in I}$ of random variables on $(\ca,\vp)$ is `characterized' by the set of all of its \emph{joint moments} $\vp(X^{i_1} \ldots X^{i_r})$ ($i_1,\ldots,i_r \in I$, $r\in \N$), and we say that $\{X^i_k\}_{i\in I}$ converges to $\{X^i\}_{i\in I}$ (when $k\to\infty$) if the convergence of the joint moments holds true (see \cite[Lecture 4]{nicaspeicher} for further details on non-commutative random systems).

\smallskip

It turns out that a rather sophisticated combinatorial machinery is hidden behind most of these objects, see \cite{nicaspeicher}. This leads in particular to the notion of \emph{crossing/non crossing pairing}, which is a central
tool in the theory.

\begin{definition}\label{d:crossing}
1. Let $r$ be an even integer. A \emph{pairing} of $\{1,\ldots,r\}$ is any partition of $\{1,\ldots,r\}$ into $r/2$ disjoint subsets, each of cardinality $2$. We denote by $\mathcal{P}_2(\{1,\ldots,r\})$ the set of all pairings of $\{1,\ldots,r\}$.

2. When $\pi \in \mathcal{P}_2(\{1,\ldots,r\})$, a \emph{crossing} in $\pi$ is any set of the form $\{\{x_1,y_1\},\{x_2,y_2\}\}$ with $\{x_i,y_i\}\in \pi$ and $x_1 < x_2 <y_1 <y_2$. The number of such crossings is denoted by $\mathrm{Cr}(\pi)$. The subset of all non-crossing pairings in $\mathcal{P}_2(\{1,\ldots,r\})$
(i.e., the subset of all $\pi\in\mathcal{P}_2(\{1,\ldots,r\})$ satisfying $\mathrm{Cr}(\pi)=0$)  is denoted by $NC_2(\{1,\ldots,r\})$.
\end{definition}

By means of the objects given in Definition \ref{d:crossing}, it is simple to compute the joint moments related to the classical Brownian motion $W$ or to the free Brownian motion $S$, and this actually leads to the so-called Wick formula. Namely, for every $t_1,\ldots,t_r \geq 0$, one has
\begin{eqnarray}\label{form-mb-class}
E\big[W_{t_1} \ldots W_{t_r}\big]&=&\sum_{\pi \in \mathcal{P}_2(\{1,\ldots,r\})} \prod_{\{i,j\}\in \pi} (t_i \wedge t_j),\\
\label{form-mb-lib}
\vp\big(S_{t_1}\ldots S_{t_r}  \big) &=&\sum_{\pi \in NC_2(\{1,\ldots,r\})} \prod_{\{i,j\}\in \pi} (t_i \wedge t_j).
\end{eqnarray}

It is possible to go smoothly from (\ref{form-mb-class}) to (\ref{form-mb-lib}) by using the $q$-Brownian motion, which is one of the nicest examples of non-commutative processes. 

\begin{definition}
Fix $q\in (-1,1)$. A $q$-Browian motion on some $W^\ast$-probability space $(\ca,\vp)$ is a
collection $\{X_t\}_{t\geq 0}$ of random variables on $(\ca,\vp)$ satisfying that, for every integer $r\geq $1 and
every $t_1,\ldots,t_r \geq 0$,
\begin{equation}\label{form-q-bm}
\vp\big( X_{t_1}\ldots X_{t_r}\big)=\sum_{\pi \in \mathcal{P}_2(\{1,\ldots,r\})} q^{\mathrm{Cr}(\pi)} \prod_{\{i,j\}\in \pi} (t_i \wedge t_j).
\end{equation}
\end{definition}

The existence of such a process, far from being trivial, is ensured by the following result.
\begin{theorem}[Bo$\dot{\rm z}$ejko, Speicher]\label{bs}
For every $q\in (-1,1)$, there exists a $W^\ast$-probability space $(\ca_q,\vp_q)$ and a $q$-Brownian motion $\{X^{(q)}_t\}_{t\geq 0}$ built
on it.
\end{theorem}

As is immediately seen, formula (\ref{form-q-bm}) allows to recover (\ref{form-mb-class}) by choosing $q=0$ (we
adopt the usual convention $0^0=1$).
On the other hand, although the classical Brownian motion $W$ cannot be identified with a process living on some $W^\ast$-probability space (the laws of its marginals being not compactly supported), it can legitimately be considered as the limit of $X^{(q)}$ when $q\to 1^-$. (This extension procedure can even be made rigorous by considering a larger class of non-commutative probability spaces.) As such, the family $\{X^{(q)}\}_{0\leq q<1}$
of $q$-Brownian motions with a parameter $q$ between 0 and 1 happens to be a `smooth' interpolation between $S$ and $W$.

\begin{definition}
Let $q\in (-1,1)$. For every $t\geq 0$, the distribution of $X^{(q)}_t$ is called the (centered) $q$-Gaussian law with variance $t$. We denote it by $\mathcal{G}_q(0,t)$. Otherwise stated, a given probability measure $\nu$ on $\R$ is distributed according to $\mathcal{G}_q(0,t)$ if it is compactly supported and if its moments are given by
\begin{equation}\label{mom-qbm-intro}
\int_{\R} x^{2k+1} \, d\nu(x)=0\quad\mbox{and}\quad
\int_{\R} x^{2k} \, d\nu(x)=t^{k} \sum_{\pi\in \mathcal{P}_2(\{1,\ldots,2k\})} q^{\mathrm{Cr}(\pi)}.
\end{equation}
The probability measure $\nu_q\sim\mathcal{G}_q(0,1)$ is absolutely continuous with respect to the Lebesgue measure; its density is supported by $\big[\frac{-2}{\sqrt{1-q}},\frac{2}{\sqrt{1-q}} \big]$ and is given, within this interval,
by
\[
\nu_q(dx)= \frac{1}{\pi} \sqrt{1-q} \sin \theta \prod_{n=1}^\infty (1-q^n) |1-q^n e^{2i\theta}|^2, \quad \text{where} \ x=\frac{2\cos \theta}{\sqrt{1-q}} \,\,\mbox{ with } \theta \in [0,\pi].
\]
By convention, we also set $\mathcal{G}_1(0,t)$ as being the probability measure whose density with respect
to the Lebesgue measure is given by
\[
\frac{1}{\sqrt{2\pi t}}e^{-\frac{x^2}{2t}},\quad x\in\R,
\]
that is, $\mathcal{G}_1(0,t)=\mathcal{N}(0,t)$.
\end{definition}

For every $q\in (-1,1)$, the process $X^{(q)}$ shares many similarities with the classical (resp. free) Brownian motion. For instance, it also appears as a limit process of some generalized random walks (see \cite[Theorem 0]{BSp2}). For this reason, one sometimes considers this `$q$-deformation' of $S$ and $W$ (see \cite{BSp1}).
Also, and similarly to the free Brownian motion case, the $q$-Brownian motion appears as the limit of some particular sequences of $q$-Gaussian random variables (see \cite{snia}).

\smallskip

In the seminal paper \cite{NP}, Nualart and Peccati highlighted a powerful convergence criterion for the normal approximation of sequences of multiple integrals with respect to the classical Brownian motion. From now on, we will refer to it as the \emph{Fourth Moment Theorem}. A few years later, it was extended by Kemp {\it et al.} \cite{knps} for the free Brownian motion $S$ and its multiple \emph{Wigner} integrals.

\smallskip

The question we shall solve in this paper is the following: what does the Fourth Moment Theorem become
when dealing more generally with a $q$-Brownian motion? Before stating our main result and in order to put
it into perspective, let us be more specific with the two afore-mentioned versions of the Fourth Moment Theorem
that are already known (that is, in the classical and free
Brownian motion cases).
We let $I_n^W$ (resp. $I_n^S$) denote the $n$th multiple integrals with respect to $W$ (resp. $S$), as they are constructed in \cite{nua} (resp. \cite{biane-speicher}).
The following two theorems are, respectively, the versions of the Fourth Moment Theorem in the classical case ($q=1$)
and in the free case ($q=0$).

\begin{theorem}[Nualart, Peccati]\label{theo-class}
Fix $n\geq 2$ and let $\{f_k\}_{k\geq 1}$ be a sequence of symmetric functions in $L^2(\R_+^n)$ satisfying
\[
E[I_n^W(f_k)^2]=n!  \|f_k\|_{L^2(\R_+^n)}^2 \to 1\quad\mbox{as $k\to\infty$}.
\]
Then, the  following two assertions are equivalent as $k\to\infty$:

\smallskip

\noindent
(i) $E\big[ I_n^W(f_k)^4 \big] \to 3$.

\smallskip

\noindent
(ii) The sequence $I_n^W(f_k)$ converges in law to $\mathcal{N}(0,1)=\mathcal{G}_1(0,1)$.
\end{theorem}

\begin{theorem}[Kemp, Nourdin, Peccati, Speicher]\label{theo-free}
Fix $n\geq 2$ and let $\{f_k\}_{k\geq 1}$ be a sequence of mirror-symmetric functions in $L^2(\R_+^n)$ (that is,
each $f_k$ is such that $f_k(t_1,\ldots,t_n)=f_k(t_n,\ldots,t_1)$ for almost all $t_1,\ldots,t_n\geq 0$) satisfying
\[
\vp\big(I_n^S(f_k)^2\big)= \|f_k\|_{L^2(\R_+^n)}^2 \to 1\quad\mbox{as $k\to\infty$}.
\]
Then, the following two assertions are equivalent as $k\to\infty$:

\smallskip

\noindent
(i) $\vp\big( I_n^S(f_k)^4 \big) \to 2$.

\smallskip

\noindent
(ii) The sequence $I_n^S(f_k)$ converges in law to $\mathcal{S}(0,1)=\mathcal{G}_0(0,1)$.
\end{theorem}

\

Now, fix a parameter $q\in[0,1]$, and consider a $q$-Brownian motion $X^{(q)}$ on some $W^\ast$-probability space $(\ca_q,\vp_q)$. (Note that in the three forthcoming statements, we extend the definition of $X^{(q)}$ to $q=1$ by naturally setting $X^{(1)}:=W$ and by replacing
$(\ca_1,\vp_1)$ by $(\Omega,\mathcal{F},P)$.) As in the classical and free cases, to each $n\geq 0$ we may associate with $X^{(q)}$ a natural notion of $n$th multiple integral $I_n^{X^{(q)}}$, see Donati-Martin \cite{don-mar} or Section \ref{subsec:mult-int}
for the details. We are now in a position to state the main result of the present paper, which is a
suitable interpolation between Theorem \ref{theo-class}
and Theorem \ref{theo-free}, but in a somehow unexpected way
(see indeed the comment following its statement).

\begin{theorem}\label{theo:intro}
Fix $n\geq 1$, recall that $q\in[0,1]$, and let $\{f_k\}_{k\geq 1}$ be a sequence of symmetric functions in $L^2(\R_+^n)$ satisfying
\[
\vp_q\big(I_n^{X^{(q)}}(f_k)^2\big)= \left( \sum_{\si\in \mathfrak{S}_n} q^{\mathrm{inv}(\si)} \right) \|f_k\|_{L^2(\R^n)}^2 \to 1 \quad \text{as $k\to\infty$},
\]
where the notation ${\rm inv}(\si)$ refers to the number of inversions in $\si$, i.e.,
$$\mathrm{inv}(\si):=\mathrm{Card}\{1\leq i< j\leq n: \ \si(i)>\si(j)\}.$$
Then the following two assertions are equivalent as $k\to\infty$:

\smallskip

\noindent
(i) $\vp_q\big( I_n^{X^{(q)}}(f_k)^4 \big) \to 2+q^{n^2}$.

\smallskip

\noindent
(ii) The sequence $I_n^{X^{(q)}}(f_k)$ converges in law to $\mathcal{G}_{q^{n^2}}(0,1)$.
\end{theorem}

\

When, in Theorem \ref{theo:intro}, we consider a value of $q$ which is {\it strictly} between 0 and 1,
 we get that any suitably-normalized sequence
$\{I_n^{X^{(q)}}(f_k)\}$ satisfying the fourth moment condition $(i)$
converges in law, see $(ii)$, to a random variable
which is expressed by means of the parameter $q^{n^2}$ and not $q$, as could have been legitimately expected
 by trying to guess the right statement with the help of Theorems \ref{theo-class} and \ref{theo-free}.
But this phenomenon was of course impossible to predict by taking a look at the case where $q\in\{0,1\}$ because, for these two values,
we precisely have that $q=q^{n^2}$.

\

Two natural questions emerge from Theorem \ref{theo:intro}: $(a)$ what can be said when $q\in(-1,0)$?
$(b)$ what happens if the functions $f_k$ are only mirror-symmetric (as in Theorem \ref{theo-free})?
Regarding $(a)$, it is not difficult to build
explicit counterexamples where the equivalence between $(i)$ and $(ii)$ in Theorem \ref{theo:intro} fails.
For instance, with $q=-1/2$, $n=2$ and $f_k=f=\sqrt{2}\,{\bf 1}_{[0,1]^2}$, we have
$\varphi_q(I_2^{X^{(q)}}(f)^2)=1$ and
$\varphi_q(I_2^{X^{(q)}}(f)^4)=2+q^4$, but $I_2^{X^{(q)}}(f)$ is not $\mathcal{G}_{q^4}(0,1)$-distributed (since $\varphi_q(I_2^{X^{(q)}}(f)^3)=\sqrt{2}(1+q)^2\neq 0$). See Remark \ref{rk:2} for the details. Actually, we do not know if this counterexample hides a general phenomenon or not. Does Theorem \ref{theo:intro} continue to be true for all $q<0$ except (possibly) for some values of $q$, or is it always false when $q<0$? On the other hand, to answer question $(b)$ is unfortunately out of the scope of this paper. Indeed, to do so would imply to change almost all our computations (in order to take into account the lack of full symmetry). We postpone this further
analysis to another paper.

\

Theorem \ref{theo:intro} will be obtained in Section \ref{section:proof} as a consequence of a more general multidimensional version (namely, Theorem \ref{theo:main}). In fact, we will even prove that $(i)$ and $(ii)$  are both equivalent to a third assertion that
only involves the
sequence $\{f_k\}_{k\geq 1}$ and not the value of $q$ (provided it belongs to $[0,1]$). As a consequence, we shall deduce the following transfer principle.

\begin{theorem}[Transfer principle]\label{transferprinciple}
Fix $n\geq 1$ and let $\{f_k\}_{k\geq 1}$ be a sequence of symmetric functions in $L^2(\R_+^n)$ satisfying $\|f_k\|_{L^2(\R_+^n)}^2 \to 1$ as $k\to\infty$. For every $q\in [0,1]$, set
\[
\si_{q}^2:=\sum_{\si\in \mathfrak{S}_n} q^{\mathrm{inv}(\si)}>0.
\]
Then, the following two assertions are equivalent as $k\to\infty$:

\smallskip

\noindent
(i) The sequence $I_n^{X^{(q)}}(f_k)$ converges in law to $\mathcal{G}_{q^{n^2}}(0,\si_{q}^2)$
for {\rm one particular} $q\in[0,1]$.

\smallskip

\noindent
(ii) The sequence $I_n^{X^{(q)}}(f_k)$ converges in law to $\mathcal{G}_{q^{n^2}}(0,\si_{q}^2)$
for {\rm all} $q\in[0,1]$.
\end{theorem}

\

As a nice application of all the previous material, we offer the following theorem. (We will prove it in Section 3.)
For every $q\in [0,1]$, let us denote by $H^{(q)}_0,H^{(q)}_1,\ldots$ the sequence of \emph{$q$-Hermite polynomials}, determined by the recurrence
$$H_0^{(q)}(x)=1, \quad H_1^{(q)}(x)=x \quad \text{and} \quad xH_n^{(q)}(x)=H^{(q)}_{n+1}(x)+[n]_q H^{(q)}_{n-1},$$
where $[n]_q=\frac{1-q^n}{1-q}$ (with the convention that $[n]_1=n$). These polynomials are related to the $q$-Brownian motion $X^{(q)}$ through the formula
\begin{equation}\label{lien-tch}
H^{(q)}_n\big( I_1^{X^{(q)}}(e)\big)=I_n^{X^{(q)}}\big( e^{\otimes n} \big),\quad e\in L^2(\R_+),\quad \|e\|_{L^2(\R_+)}^2=1.
\end{equation}
We then have:

\

\begin{theorem}[$q$-version of the Breuer-Major theorem]\label{breuermajor}
Fix $q\in [0,1]$ and let $n\geq 1$.
Let $\{G_l\}_{l\in\N}$ be a $q$-Gaussian centered stationary family of random variables on some $W^\ast$-probability space $(\ca,\vp)$, meaning that there exists $\rho:\Z\to\R$ such that, for every integer $r\geq 1$ and every $l_1,\ldots,l_r\geq 1$, one has
\[
\vp\big( G_{l_1}\ldots G_{l_r}\big)=\sum_{\pi \in \mathcal{P}_2(\{1,\ldots,r\})} q^{\mathrm{Cr}(\pi)} \prod_{\{a,b\}\in \pi} \rho(l_a-l_b).
\]
Assume further that $\rho(0)=1$ (this just means that $G_l \sim \mathcal{G}_q(0,1)$ for every $l$) and
$\sum_{l\in \Z} |\rho(l)|^n$ is finite. Then, as $k\to\infty$,
\begin{equation}\label{thm19expr}
\left\{\frac{1}{\sqrt{k}} \sum_{l=0}^{[kt]} H_n^{(q)}(G_l)\right\}_{t\geq 0} \,\overset{\rm f.d.d.}{\to}\, \sqrt{\sum_{\si\in \mathfrak{S}_n} q^{\mathrm{inv}(\si)} \sum_{l\in \Z} \rho(l)^n}\,\left\{X_t^{(q^{n^2})}\right\}_{t\geq 0},
\end{equation}
where
`f.d.d.' stands for the convergence in law of all finite-dimensional distributions and
$X^{(q^{n^2})}$ is a $q^{n^2}$-Brownian motion.
\end{theorem}

The rest of the paper is divided into two sections. In Section \ref{section:2}, we recall and prove some useful results relative to the so-called \emph{$q$-Gaussian chaos}, which is nothing but a generalization of both the Wiener and Wigner chaoses. Notably, therein
we extend the formula (\ref{form-q-bm}) to the case of multiple integrals with respect to the $q$-Brownian motion (Theorem \ref{thm:form-mom}). Once endowed with this preliminary material, we devote Section \ref{section:proof} to the proofs of Theorems \ref{theo:intro}, \ref{transferprinciple} and \ref{breuermajor}.

\section{$q$-Brownian chaos and product formulae}\label{section:2}

Throughout this section, we fix a parameter $q\in (-1,1)$, as well as a $q$-Brownian motion $X^{(q)}$ on some $W^\ast$-probability space
$(\ca_q,\vp_q)$.
As a first step towards Theorem \ref{theo:intro}, our aim is to generalize the formula (\ref{form-q-bm}) to the case of multiple integrals with respect to $X^{(q)}$.

\subsection{Multiple integrals}\label{subsec:mult-int}

For every integer $n\geq 1$, the collection of all random variables of the type \[
I_n^{X^{(q)}}(f)=\int_{\R^n_+} f(t_1,\ldots,t_n) \,  dX^{(q)}_{t_1} \ldots dX^{(q)}_{t_n},\quad
f \in L^2(\mathbb{R}_+^n),
\]
is called the $n$th {\it $q$-Gaussian chaos} associated with $X^{(q)}$, and has been defined by Donati-Martin \cite{don-mar} along the same lines as the classical Wiener chaos (see, e.g., \cite{nua}), namely:

\smallskip

\noindent
- first define $I^{X^{(q)}}_n(f) = (X^{(q)}_{b_1} - X^{(q)}_{a_1})\ldots (X^{(q)}_{b_n} - X^{(q)}_{a_n})$ when $f$ has the form
\begin{equation}\label{e:simple}
f(t_1,...,t_n) = {\bf 1}_{(a_1,b_1)}(t_1)\times\ldots\times {\bf 1}_{(a_n,b_n)}(t_n),
\end{equation}
where the intervals $(a_i,b_i)$, $i=1,...,n$, are pairwise disjoint;

\smallskip

\noindent
- extend linearly the definition of $I^{X^{(q)}}_n(f)$ to the class $\mathcal{E}$ of simple functions vanishing on diagonals, that is, to functions $f$ that are finite
linear combinations of indicators of the type (\ref{e:simple});

\smallskip

\noindent
- observe that, for all simple functions $f\in L^2(\R_+^m)$ and $g\in L^2(\R_+^n)$ vanishing on diagonals,
\begin{equation}\label{isometrie}
\langle I^{X^{(q)}}_m(f),I^{X^{(q)}}_n(g) \rangle_{L^2(\mathcal{A}_q,\vp_q)}=
\varphi_q\left(I^{X^{(q)}}_m(f)^*I^{X^{(q)}}_n(g)\right)=\der_{m,n}
 \langle f,g \rangle_{q},
\end{equation}
where the sesquilinear form $\langle .,.\rangle_q$ is defined for all $f,g\in L^2(\R^n_+)$ by
\begin{equation}\label{prod-scal}
\langle f,g\rangle_q:=\sum_{\si \in \mathfrak{S}_n} q^{\mathrm{inv}(\si)}\int_{\R_+^n} f(t_{\si(1)},\ldots,t_{\si(n)})g(t_1,\ldots,t_n) \, dt_1 \ldots dt_n
\end{equation}
and where $\delta_{m,n}$ stands for the Kronecker symbol;

\smallskip

\noindent
- exploit the fact that the form $\langle .,.\rangle_q$ is strictly positive on $L^2(\R^n_+)$ (see \cite[Proposition 1]{BSp1}) in order to extend $I_n^{X^{(q)}}(f)$ to functions $f$ in the completion $\mathcal{F}_q$ of $\mathcal{E}$ with respect to $\langle .,.\rangle_q$. Observe finally that, owing to the estimate $\|f\|^2_q \leq \big( \sum_{\si\in \mathfrak{S}_n} q^{\mathrm{inv}(\si)} \big) \|f\|^2_{L^2(\R^n_+)}$, one can rely on the inclusion $L^2(\R^n_+) \subset \mathcal{F}_q$ for every $q\in (-1,1)$ and every $n\geq 1$.

\

Of course, relation (\ref{isometrie}) continues to hold for every pair $f \in L^2(\mathbb{R}^n_+)$  and $g\in L^2(\R_+^n)$. Moreover, the above sketched
construction implies that $I^{X^{(q)}}_n(f)$ is self-adjoint if and only if $f$ is \emph{mirror symmetric}, i.e., $f^\ast=f$ where $f^\ast(t_1,\ldots,t_n):=f(t_n,\ldots,t_1)$.

\

Let us now report one of the main results of \cite{don-mar}, namely the generalization of the product formula for multiple Wiener-It\^o integrals  to the $q$-Brownian motion case.
In the sequel, we adopt the following notation.

\

\textbf{Notation.} With every $f\in L^2(\R_+^n)$ and every $p\in \{1,\ldots,n\}$, we associate the function $f^{(p)}_q\in L^2(\R_+^n)$ along the formula
$$f^{(p)}_q(t_1,\ldots,t_{n-p},s_p,\ldots,s_1):=\sum_{\si:\{1,\ldots,p\} \to \{1,\ldots,n\} \searrow} q^{\al(\si)} f(t_1,\ldots,s_k,\ldots,s_1,\ldots,t_{n-k}),$$
where, in the right-hand-side, $\si$ is decreasing (this fact is written in symbols as $\sigma\searrow$), $s_i$ is at the place $\si(i)$, and
$$\al(\si):=\sum_{i=1}^p (n+1-\si(i))-\frac{p(p+1)}{2}.$$
Besides, we define another function $f^{[p]}_q\in L^2(\R_+^n)$ by
$$f^{[p]}_q(s_1,\ldots,s_p,t_1,\ldots,t_{n-k}):=\sum_{\si:\{1,\ldots,p\} \to \{1,\ldots,n\}} q^{\beta(\si)}f(t_1,\ldots,s_i,\ldots t_{n-p}),$$
where, in the right-hand-side, $s_i$ is at the place $\si(i)$ and
$$\beta(\si):=\sum_{i=1}^p \si(i)-\frac{p(p+1)}{2}+\mathrm{inv}(\si).$$
(See Theorem \ref{theo:intro} for the definition of $\mathrm{inv}(\si)$.)

\

We now introduce the central concept of {\it contractions}.

\begin{definition}
Fix $n,m\geq 1$ as well as $p\in \{1,\ldots,\min(m,n)\}$. Let $f\in L^2(\R_+^n)$ and $g\in L^2(\R_+^m)$.\\
1.  The $p$th contraction $f\contra{p}g \in L^2(\R_+^{m+n-2p})$ of $f$ and $g$ is defined by the formula
\begin{multline*}
f\contra{p} g(t_1,\ldots,t_{m+n-2p})\\=\int_{\R_+^p}  f(t_1,\ldots,t_{n-p},s_p,\ldots,s_1)g(s_1,\ldots,s_p,t_{n-p+1},\ldots,t_{m+n-2p})ds_1 \ldots ds_p.
\end{multline*}
2. The $p$th $q$-contraction $f \pt{p}g \in L^2(\R_+^{m+n-2p})$ of $f$ and $g$ is defined by the formula
\[
f\pt{p} g=f^{(p)}_q \contra{p} g^{[p]}_q.
\]
3. We also set $f\pt{0} g=f\contra{0}g=f\otimes g$.
\end{definition}

These contractions appear naturally in the product formula for multiple integrals with respect to the $q$-Brownian motion, that we
state now.

\begin{theorem}[Donati-Martin] Let $f\in L^2(\R_+^n)$ and $g\in L^2(\R_+^m)$ with $n,m\geq 1$. Then
\begin{equation}\label{mul}
I_n^{X^{(q)}}(f)I_m^{X^{(q)}}(g)=\sum_{p=0}^{\min(n,m)} I_{n+m-2p}^{X^{(q)}}\big(f\pt{p} g\big).
\end{equation}
\end{theorem}

\subsection{Respecting pairings}\label{subsec:pairings}

As in \cite{knps}, the notion of a \emph{respecting pairing} will play a prominent role in our study.

\begin{definition}
Let $n_1,\ldots,n_r$ be positive integers and $n=n_1+\ldots+n_r$. The set $\{1,\ldots,n\}$ is then partitioned accordingly as $\{1,\ldots,n\}=B_1 \cup B_2 \cup \ldots \cup B_r$, where $B_1=\{1,\ldots,n_1\}$, $B_2=\{n_1+1,\ldots,n_1+n_2\}$, $\ldots$, $B_r=\{n_1+\ldots+n_{r_1}+1,\ldots,n\}$. We denote this partition by $n_1 \otimes \ldots \otimes n_r$, and we will refer to the sets $B_i$ as the \emph{blocks} of $n_1\otimes \ldots \otimes n_r$.\\
Then, we say that a pairing $\pi \in \mathcal{P}_2(\{1,\ldots,n\})$ \emph{respects} $n_1\otimes \ldots \otimes n_r$ if every pair $\{l,m\}\in \pi$ is such that $l\in B_i$ and $m\in B_j$ with $i\neq j$. In the sequel, the subset of such respecting pairings in $\mathcal{P}_2(\{1,\ldots,n\})$ will be denoted as $C_2(n_1\otimes \ldots \otimes n_r)$.\\
Finally, given $\pi \in C_2(n_1\otimes \ldots \otimes n_r)$ and functions $f^1\in L^2(\R_+^{n_1}),\ldots,f^r \in L^2(\R_+^{n_r})$, we define the \emph{pairing integral}
\begin{multline}\label{pairing-int}
\int_\pi f_1 \otimes \ldots \otimes f_r
:=\int_{\R_+^n} dt_1 \ldots dt_n\\ \times f_1(t_1,\ldots,t_{n_1}) f_2(t_{n_1+1},\ldots,t_{n_1+n_2})\ldots f_r(t_{n_1+\ldots+n_{r-1}+1}, \ldots,t_n)\prod_{\{i,j\}\in \pi} \der(t_i-t_j),
\end{multline}
where $\der$ stands for a Dirac mass at $0$.
\end{definition}

For instance, consider the following pairing
$$\pi:=\{(1,4),(2,8),(3,6),(5,9),(7,11),(10,12)\}$$
as an element of $C_2(3\otimes 4 \otimes 3 \otimes 2)$. Then it is readily checked that
$$\int_\pi f_1 \otimes f_2 \otimes f_3 \otimes f_4=\int_{\R_+^6}  f_1(t_1,t_2,t_3)f_2(t_1,t_4,t_3,t_5)f_3(t_2,t_4,t_6)f_4(t_5,t_6) dt_1dt_2dt_3dt_4dt_5 dt_6.$$

\begin{lemma}
Let $f,g\in L^2(\R_+^n)$ and
recall the definition (\ref{prod-scal}) of $\langle f,g\rangle_q$.
We have
\begin{equation}\label{prod-scal-2}
\langle f,g\rangle_q=\sum_{\si \in \mathfrak{S}_n} q^{\mathrm{inv}(\si)}\int_{P_2(\si)} f\otimes g^\ast
\end{equation}
where, for every $\si\in \mathfrak{S}_n$, the pairing $P_2(\si) \in C_2(n\otimes n)$ is explicitly given by
$$P_2(\si):=\{(n+1-i,n+\si(i)), \ 1\leq i\leq n\}.$$
\end{lemma}
\begin{proof}
With each $\si \in \mathfrak{S}_n$, we may associate $\widetilde{\sigma}\in \mathfrak{S}_n$ given by $\widetilde{\sigma}(i)=n+1-\si(n+1-i)$.
We then have, by the definition (\ref{pairing-int}),
\begin{eqnarray*}
\int_{P_2(\widetilde{\sigma})} f\otimes g^\ast
&=&\int_{\R_+^n} f(s_1,\ldots,s_{n}) g^\ast(s_{n+1-\widetilde{\sigma}^{-1}(1)},\ldots,s_{n+1-\widetilde{\sigma}^{-1}(n)})ds_1 \ldots ds_n\\
&=&\int_{\R_+^n} f(s_1,\ldots,s_{n}) g(s_{n+1-\widetilde{\sigma}^{-1}(n)},\ldots,s_{n+1-\widetilde{\sigma}^{-1}(1)})ds_1 \ldots ds_n\\
&=&\int_{\R_+^n} f(t_{\si(1)},\ldots,t_{\si(n)}) g(t_{\si(n+1-\widetilde{\sigma}^{-1}(n))},\ldots,t_{\si(n+1-\widetilde{\sigma}^{-1}(1))})dt_1 \ldots dt_n.
\end{eqnarray*}
Now, we observe that $\si(n+1-\widetilde{\si}^{-1}(i))=n+1-\widetilde{\si}(\widetilde{\si}^{-1}(i))=n+1-i$, so that $\si(n+1-\widetilde{\si}^{-1}(n+1-i))=i$ for any $i$.
We deduce that
\[
\int_{P_2(\widetilde{\sigma})} f\otimes g^\ast
=\int_{\R_+^n} f(t_{\si(1)},\ldots,t_{\si(n)}) g(t_{1},\ldots,t_{n})dt_1 \ldots dt_n.
\]
Thus, since it is further readily checked that $\mathrm{inv}(\widetilde{\sigma})=\mathrm{inv}(\si)$
and that $\si\mapsto\widetilde{\si}$ is an involution, we get
\begin{eqnarray*}
\sum_{\si \in \mathfrak{S}_n} q^{\mathrm{inv}(\si)}\int_{P_2(\si)} f\otimes g^\ast &=&\sum_{\si \in \mathfrak{S}_n} q^{\mathrm{inv}(\si)}\int_{P_2(\widetilde{\si})} f\otimes g^\ast\\
& =&
\sum_{\si \in \mathfrak{S}_n} q^{\mathrm{inv}(\si)}\int_{\R_+^n} f(t_{\si(1)},\ldots,t_{\si(n)}) g(t_{1},\ldots,t_{n})dt_1 \ldots dt_n\\
&=&\langle f,g\rangle_q.
\end{eqnarray*}
\end{proof}

\subsection{Joint moments of multiple integrals}\label{subsec:mom}

Let us eventually turn to the main concern of this section, that is, to the extension of (\ref{form-q-bm}) for multiple integrals $I_{n_1}^{X^{(q)}}(f^1),\ldots,I_{n_r}^{X^{(q)}}(f^r)$. To achieve this goal, we focus on the following construction procedure.

\smallskip

Fix some positive integers $n_1,\ldots,n_r$ with $r\geq 3$, as well as $p\in \{1,\ldots,\min(n_1,n_2)\}$. Then, given $\pi'\in C_2((n_1+n_2-2p)\otimes n_3 \otimes \ldots \otimes n_r)$, $\si_1:\{1,\ldots, p\} \to \{1,\ldots,n_1\} \searrow$ and $\si_2:\{1,\ldots,p\} \to \{1,\ldots,n_2\}$, we construct a pairing $\pi =F(\si_1,\si_2,\pi')\in C_2(n_1\otimes n_2\otimes \ldots \otimes n_r) $ as follows (see Figure 1
for an illustration):\\

\noindent
1) In $\pi$, the first two blocks $\{1,\ldots,n_1\}$ and $\{n_1+1,\ldots,n_1+n_2\}$ are connected via exactly $p$ pairs given by
$$\big\{ (\si_1(i),n_1+\si_2(i)), \ 1\leq i\leq p\big\}.$$
2) The interactions between the $n_1+n_2-2p$ remaining points in $\{1,\ldots,n_1+n_2\}$ and the set $\{n_1+n_2+1,\ldots , n_1+\ldots, n_r\}$, as well as the interactions within $\{n_1+n_2+1,\ldots , n_1+\ldots, n_r\}$, are governed along $\pi'$.

\begin{center}
\begin{figure}[!ht]\label{pic:constr}

\begin{pspicture}(0,0)(12.8,13)


\psline(-0.2,4.8)(3.2,4.8) \psline(-0.2,4.8)(-0.2,5) \psline(3.2,4.8)(3.2,5)

\psline(3.8,4.8)(6,4.8) \psline(3.8,4.8)(3.8,5) \psline(6,4.8)(6,5)

\psline(6.6,4.8)(8.8,4.8) \psline(6.6,4.8)(6.6,5) \psline(8.8,4.8)(8.8,5)

\psline(-0.2,-0.2)(3.2,-0.2) \psline(-0.2,-0.2)(-0.2,0) \psline(3.2,-0.2)(3.2,0)

\psline(3.8,-0.2)(7.2,-0.2) \psline(3.8,-0.2)(3.8,0) \psline(7.2,-0.2)(7.2,0)

\psline(7.8,-0.2)(10,-0.2) \psline(7.8,-0.2)(7.8,0) \psline(10,-0.2)(10,0)

\psline(10.6,-0.2)(12.8,-0.2) \psline(10.6,-0.2)(10.6,0) \psline(12.8,-0.2)(12.8,0)

\psline(-0.2,9.4)(5.2,9.4) \psline(-0.2,9.4)(-0.2,9.6) \psline(5.2,9.4)(5.2,9.6)

\psline(6.8,9.4)(12.2,9.4) \psline(6.8,9.4)(6.8,9.6) \psline(12.2,9.4)(12.2,9.6)

\psline[linecolor=blue](1,10)(1,13) \psline[linecolor=blue](1,13)(9,13) \psline[linecolor=blue](9,13)(9,10)

\psline[linecolor=blue](3,10)(3,12) \psline[linecolor=blue](3,12)(7,12) \psline[linecolor=blue](7,12)(7,10)

\psline[linecolor=blue](5,10)(5,11) \psline[linecolor=blue](5,11)(11,11) \psline[linecolor=blue](11,11)(11,10)

\rput(1,9.8){\tiny{$\si_1(3)$}} \rput(3,9.8){\tiny{$\si_1(2)$}} \rput(5,9.8){\tiny{$\si_1(1)$}}

\rput(9,9.8){\tiny{$6+\si_2(3)$}} \rput(7,9.8){\tiny{$6+\si_2(2)$}} \rput(11,9.8){\tiny{$6+\si_2(1)$}}

\rput(0,10){$\bullet$} \rput(1,10){$\bullet$} \rput(2,10){$\bullet$} \rput(3,10){$\bullet$} \rput(4,10){$\bullet$} \rput(5,10){$\bullet$}
\rput(7,10){$\bullet$} \rput(8,10){$\bullet$} \rput(9,10){$\bullet$} \rput(10,10){$\bullet$} \rput(11,10){$\bullet$} \rput(12,10){$\bullet$}

\psline[linecolor=green](0,5)(0,8) \psline[linecolor=green](0,8)(7.4,8) \psline[linecolor=green](7.4,8)(7.4,5)

\psline[linecolor=green](0.6,5)(0.6,7.6) \psline[linecolor=green](0.6,7.6)(4.6,7.6) \psline[linecolor=green](4.6,7.6)(4.6,5)

\psline[linecolor=green](1.2,5)(1.2,7.2) \psline[linecolor=green](1.2,7.2)(5.8,7.2) \psline[linecolor=green](5.8,7.2)(5.8,5)

\psline[linecolor=green](1.8,5)(1.8,6.8) \psline[linecolor=green](1.8,6.8)(8,6.8) \psline[linecolor=green](8,6.8)(8,5)

\psline[linecolor=green](2.4,5)(2.4,6.4) \psline[linecolor=green](2.4,6.4)(8.6,6.4) \psline[linecolor=green](8.6,6.4)(8.6,5)

\psline[linecolor=green](3,5)(3,6) \psline[linecolor=green](3,6)(4,6) \psline[linecolor=green](4,6)(4,5)

\psline[linecolor=green](5.2,5)(5.2,6) \psline[linecolor=green](5.2,6)(6.8,6) \psline[linecolor=green](6.8,6)(6.8,5)

\rput(0,5){$\bullet$} \rput(0.6,5){$\bullet$} \rput(1.2,5){$\bullet$} \rput(1.8,5){$\bullet$} \rput(2.4,5){$\bullet$}
\rput(3,5){$\bullet$} \rput(4,5){$\bullet$} \rput(4.6,5){$\bullet$} \rput(5.2,5){$\bullet$} \rput(5.8,5){$\bullet$}
\rput(6.8,5){$\bullet$} \rput(7.4,5){$\bullet$} \rput(8,5){$\bullet$} \rput(8.6,5){$\bullet$}

\psline[linecolor=green](0,0)(0,4) \psline[linecolor=green](0,4)(11.4,4) \psline[linecolor=green](11.4,4)(11.4,0)

\psline[linecolor=blue](0.6,0)(0.6,3.6) \psline[linecolor=blue](0.6,3.6)(5.2,3.6) \psline[linecolor=blue](5.2,3.6)(5.2,0)

\psline[linecolor=green](1.2,0)(1.2,3.2) \psline[linecolor=green](1.2,3.2)(8.6,3.2) \psline[linecolor=green](8.6,3.2)(8.6,0)

\psline[linecolor=blue](1.8,0)(1.8,2.8) \psline[linecolor=blue](1.8,2.8)(4,2.8) \psline[linecolor=blue](4,2.8)(4,0)

\psline[linecolor=green](2.4,0)(2.4,2.4) \psline[linecolor=green](2.4,2.4)(9.8,2.4) \psline[linecolor=green](9.8,2.4)(9.8,0)

\psline[linecolor=blue](3,0)(3,2) \psline[linecolor=blue](3,2)(6.4,2) \psline[linecolor=blue](6.4,2)(6.4,0)

\psline[linecolor=green](4.6,0)(4.6,1.6) \psline[linecolor=green](4.6,1.6)(12,1.6) \psline[linecolor=green](12,1.6)(12,0)

\psline[linecolor=green](5.8,0)(5.8,1.2) \psline[linecolor=green](5.8,1.2)(12.6,1.2) \psline[linecolor=green](12.6,1.2)(12.6,0)

\psline[linecolor=green](7,0)(7,0.8) \psline[linecolor=green](7,0.8)(8,0.8) \psline[linecolor=green](8,0.8)(8,0)

\psline[linecolor=green](9.2,0)(9.2,0.8) \psline[linecolor=green](9.2,0.8)(10.8,0.8) \psline[linecolor=green](10.8,0.8)(10.8,0)

\rput(0,0){$\bullet$} \rput(0.6,0){$\bullet$} \rput(1.2,0){$\bullet$} \rput(1.8,0){$\bullet$} \rput(2.4,0){$\bullet$} \rput(3,0){$\bullet$}
\rput(4,0){$\bullet$} \rput(4.6,0){$\bullet$} \rput(5.2,0){$\bullet$} \rput(5.8,0){$\bullet$} \rput(6.4,0){$\bullet$} \rput(7,0){$\bullet$}
\rput(8,0){$\bullet$} \rput(8.6,0){$\bullet$} \rput(9.2,0){$\bullet$} \rput(9.8,0){$\bullet$}
\rput(10.8,0){$\bullet$} \rput(11.4,0){$\bullet$} \rput(12,0){$\bullet$} \rput(12.6,0){$\bullet$}

\end{pspicture}

\caption{Construction of a pairing $\pi=F(\si_1,\si_2,\pi')\in C_2(6\otimes 6 \otimes 4 \otimes 4)$ (third graph) from $\si_1:\{1,2,3\} \to \{1,\ldots,6\} \searrow$, $\si_2:\{1,2,3\} \to \{1,\ldots,6\}$ (first graph) and $\pi'\in C_2(6\otimes 4\otimes 4)$ (second graph).}
\end{figure}
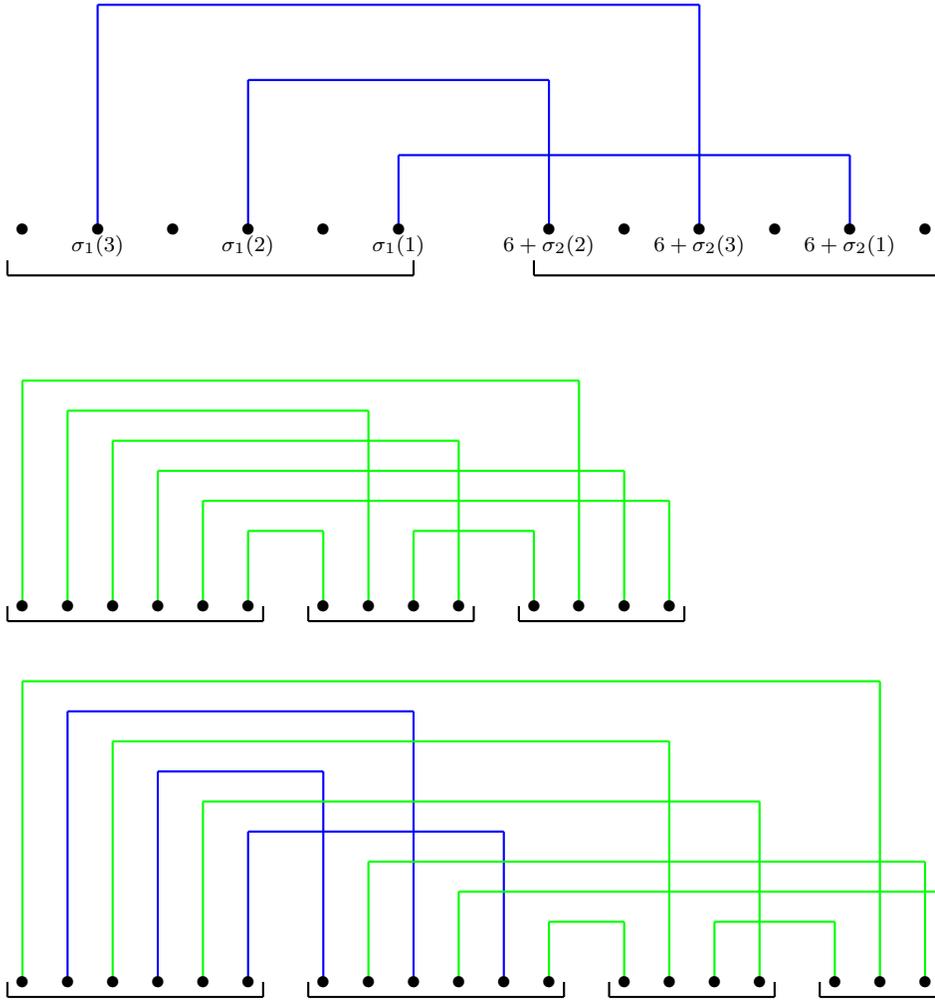
\end{center}

This construction is clearly a one-to-one procedure. That is, given a pairing $\pi \in C_2(n_1\otimes n_2\otimes \ldots \otimes n_r)$ such that the first two blocks $\{1,\ldots,n_1\}$ and $\{n_1+1,\ldots,n_1+n_2\}$ are linked by (exactly) $p$ pairs with $p\in \{1,\ldots,\min(n_1,n_2)\}$, there exists a unique $\si_1:\{1,\ldots, p\} \to \{1,\ldots,n_1\} \searrow$, a unique $\si_2:\{1,\ldots,p\} \to \{1,\ldots,n_2\}$ and a unique pairing $\pi'\in C_2((n_1+n_2-2p)\otimes n_3 \otimes \ldots \otimes n_r)$ such that $\pi=F(\si_1,\si_2,\pi')$. Besides, by the very definition of the $q$-contraction $f\pt{p}g$, the following result is easily checked:

\begin{lemma}
Fix $p\in \{1,\ldots,\min(n_1,n_2)\}$ and $\pi' \in C_2((n_1+n_2-2p)\otimes n_3 \otimes \ldots \otimes n_r)$. Then, for all functions $f^1\in L^2(\R_+^{n_1}),\ldots,f^r \in L^2(\R_+^{n_r})$, one has
\begin{equation}\label{decomp}
\int_{\pi'} \big( f_1 \pt{p} f_2\big) \otimes f_3 \otimes \ldots \otimes f_r=\sum_{\substack{\si_1:\{1,\ldots,p\} \to \{1,\ldots,n_1\} \searrow\\ \si_2:\{1,\ldots,p\} \to \{1,\ldots,n_2\}}} q^{\al(\si_1)+\beta(\si_2)} \int_{F(\si_1,\si_2,\pi')} f_1\otimes f_2\otimes \ldots \otimes f_r.
\end{equation}
\end{lemma}

Our second ingredient for the generalization of (\ref{form-q-bm}) lies in the following computation of the crossings in $F(\si_1,\si_2,\pi')$:

\begin{lemma}
Fix $p\in \{1,\ldots,\min(n_1,n_2)\}$, $\pi' \in C_2((n_1+n_2-2p)\otimes n_3 \otimes \ldots \otimes n_r)$, $\si_1:\{1,\ldots,p\} \to \{1,\ldots,n_1\}\searrow$ and $\si_2:\{1,\ldots,p\} \to \{1,\ldots,n_2\}$. Then
\begin{equation}\label{recu}
\mathrm{Cr}(F(\si_1,\si_2,\pi'))=\al(\si_1)+\beta(\si_2)+\mathrm{Cr}(\pi').
\end{equation}
\end{lemma}

\begin{proof}
Set $\pi:=F(\si_1,\si_2,\pi')$. The difference $D:=\text{Cr}(\pi)-\text{Cr}(\pi')$ is given by the number of crossings in $\pi$ that involve at least one of the pairs $\{(\si_1(i),n_1+\si_2(i)), \ 1\leq i\leq p\}$. In order to compute this quantity, consider the following iterative procedure:

\noindent
- \emph{Step 1}: Compute the number of crossings that involve the pair $(\si_1(1),n_1+\si_2(1))$. Since $\si_1$ is decreasing and $\pi\in C_2(n_1 \otimes \ldots \otimes n_r)$, this is just the number of points between $\si_1(1)$ and $n_1+\si_2(1)$, i.e., $(n_1+\si_2(1))-\si_1(1)-1$.

\smallskip

\noindent
- \emph{Step 2}: Compute the number of crossings that involve the pair $(\si_1(2),n_1+\si_2(2))$, leaving aside the possible crossings between $(\si_1(2),n_1+\si_2(2))$ and $(\si_1(1),n_1+\si_2(1))$ (they have already been taken into account in Step 1). Since $\si_1(1)>\si_1(2)$, this leads to $(n_1+\si_2(2))-\si_1(2)-2-1_{\{\si_2(1)<\si_2(2)\}}$ crossings.

\smallskip

\vdots

\smallskip

\noindent
- \emph{Step l}: Compute the number of crossings that involve the pair $(\si_1(l),n_1+\si_2(l))$, leaving aside the possible crossings between $(\si_1(l),n_1+\si_2(l))$ and $(\si_1(1),n_1+\si_2(1)),\ldots,(\si_1(l-1),n_1+\si_2(l-1))$ (they have already been taken into account in the previous steps). This yields  $(n_1+\si_2(l))-\si_1(l)-l-\sum_{j=1}^{l-1} 1_{\{\si_2(j)<\si_2(l)\}}$ crossings.

\smallskip

\vdots

\smallskip

\noindent
By repeating this procedure up to Step $p$, one can compute $D$ as follows:
\bean
D&=& \sum_{l=1}^p \bigg[ (n_1+\si_2(l))-\si_1(l)-l-\sum_{j=1}^{l-1} 1_{\{\si_2(j)<\si_2(l)\}} \bigg]\\
&=&\sum_{l=1}^p \big[(n_1+\si_2(l))-\si_1(l)\big] -\bigg\{ \frac{p(p+1)}{2} +\sum_{l=1}^p \sum_{j=1}^{l-1} 1_{\{\si_2(j)<\si_2(l)\}} \bigg\}\\
&=&\sum_{l=1}^p \big[(n_1+\si_2(l))-\si_1(l)\big] -\{p^2-\mathrm{inv}(\si_2)\} \\
&=& \bigg[ p(n_1+1)-\sum_{l=1}^p \si_1(l)-\frac{p(p+1)}{2}\bigg]+\bigg[ \sum_{l=1}^p \si_2(l)-\frac{p(p+1)}{2}+\mathrm{inv}(\si_2)\bigg]\\
&=& \al(\si_1)+\beta(\si_2).
\eean

\end{proof}

We are now in a position to state the main result of this section, which is a suitable generalization of (\ref{form-q-bm}). (We recover (\ref{form-q-bm}) by choosing $n_1=\ldots=n_r=1$ and
$f_i={\bf 1}_{[0,t_i]}$, $i=1,\ldots,r$.)
\begin{theorem}\label{thm:form-mom}
Let $n_1,\ldots,n_r$ be positive integers. For all functions $f_1\in L^2(\R_+^{n_1}),\ldots,f_r \in L^2(\R_+^{n_r})$, it holds that
\begin{equation}\label{form-mom}
\vp_q\big(I^{X^{(q)}}_{n_1}(f_1)\ldots I^{X^{(q)}}_{n_r}(f_r)\big)=\sum_{\pi \in C_2(n_1\otimes\ldots \otimes n_r)} q^{\mathrm{Cr}(\pi)} \int_{\pi} f_1 \otimes \ldots \otimes f_r.
\end{equation}
\end{theorem}

\begin{proof}
The proof is
by induction on $r\geq 2$. For $r=2$, observe first that $C_2(n_1 \otimes n_2)\neq \emptyset \Leftrightarrow n_1=n_2$. Then, with the notation of Section \ref{subsec:pairings}, every $\pi \in C_2(n \otimes n)$ can be written as $\pi=P_2(\si)$ for a unique $\si\in \mathfrak{S}_n$, and one has $\text{Cr}(\pi)=\text{inv}(\si)$ since, for every $i,j\in \{1,\ldots,n\}$,
$$n+1-j < n+1-i< n+\si(j)<n+\si(i) \quad \Longleftrightarrow \quad  i<j \ \text{and} \ \si(j)<\si(i).$$
Therefore, according to (\ref{prod-scal-2}) and (\ref{isometrie}),
$$\sum_{\pi \in C_2(n\otimes n)} q^{\mathrm{Cr}(\pi)} \int_\pi f_1 \otimes f_2 = \sum_{\si \in \mathfrak{S}_n} q^{\mathrm{inv}(\si)} \int_{P_2(\si)} f_1 \otimes f_2 = \langle f_1,f_2^{\ast} \rangle_q = \vp_q \big( I_n^{X^{(q)}}(f_1)I_n^{X^{(q)}}(f_2) \big),
$$
which corresponds to (\ref{form-mom}) in this case.

\smallskip

\noindent
Assume now that (\ref{form-mom}) holds true for every $(n_1,\ldots,n_s)$ with $s\leq r-1$ ($r\geq 3$), and fix $n_1,\ldots,n_r \geq 1$. According to the product formula (\ref{mul}), one can write
$$\vp_q\big(I^{X^{(q)}}_{n_1}(f_1)\ldots I^{X^{(q)}}_{n_r}(f_r)\big)=\sum_{p=0}^{\min(n_1,n_2)} \vp_q \big( I^{X^{(q)}}_{n_1+n_2-2p}(f_1 \pt{p} f_2) I^{X^{(q)}}_{n_3}(f_3) \ldots I^{X^{(q)}}_{n_r}(f_r) \big),$$
and hence, by our induction assumption,
$$\vp_q\big(I^{X^{(q)}}_{n_1}(f_1)\ldots I^{X^{(q)}}_{n_r}(f_r)\big)=\sum_{p=0}^{\min(n_1,n_2)} \sum_{\pi \in C_2((n_1+n_2-2p)\otimes n_3\otimes \ldots \otimes n_r)} q^{\mathrm{Cr}(\pi)} \int_\pi (f_1 \pt{p}f_2) \otimes f_3 \otimes \ldots \otimes f_r.$$
Now, given $\pi \in C_2(n_1 \otimes \ldots \otimes n_r)$, denote by $p_\pi\in \{0,\ldots,\min(n_1,n_2)\}$ the number of pairs in $\pi$ that link the first two blocks $\{1,\ldots,n_1\}$ and $\{n_1+1,\ldots,n_1+n_2\}$.
By using successively Formulae (\ref{decomp}) and (\ref{recu}), we deduce
\bean
\lefteqn{\vp_q\big(I^{X^{(q)}}_{n_1}(f_1)\ldots I^{X^{(q)}}_{n_r}(f_r)\big)}\\
&=& \sum_{\pi \in C_2((n_1+n_2) \otimes n_3 \otimes \ldots \otimes n_r)} q^{\mathrm{Cr}(\pi)} \int_\pi (f_1 \otimes f_2) \otimes f_3 \otimes \ldots \otimes f_r\\
& &+\sum_{p=1}^{\min(n_1,n_2)} \sum_{\substack{\pi \in C_2((n_1+n_2-2p)\otimes n_3\otimes \ldots \otimes n_r)\\ \si_1:\{1,\ldots,p\} \to \{1,\ldots,n_1\} \searrow\\ \si_2:\{1,\ldots,p\} \to \{1,\ldots,n_2\}}} q^{\al(\si_1)+\beta(\si_2)+\mathrm{Cr}(\pi)} \int_{F(\si_1,\si_2,\pi)} f_1\otimes f_2\otimes \ldots \otimes f_r\\
&=& \sum_{\substack{\pi \in C_2(n_1\otimes n_2 \otimes n_3 \otimes \ldots \otimes n_r)\\ p_\pi=0}} q^{\mathrm{Cr}(\pi)} \int_\pi f_1 \otimes f_2 \otimes f_3 \otimes \ldots \otimes f_r\\
& &+\sum_{p=1}^{\min(n_1,n_2)} \sum_{\substack{\pi \in C_2((n_1+n_2-2p)\otimes n_3\otimes \ldots \otimes n_r)\\ \si_1:\{1,\ldots,p\} \to \{1,\ldots,n_1\} \searrow\\ \si_2:\{1,\ldots,p\} \to \{1,\ldots,n_2\}}} q^{\mathrm{Cr}(F(\si_1,\si_2,\pi))} \int_{F(\si_1,\si_2,\pi)} f_1\otimes f_2\otimes \ldots \otimes f_r\\
&=& \sum_{\substack{\pi \in C_2(n_1\otimes n_2 \otimes n_3 \otimes \ldots \otimes n_r)\\ p_\pi=0}} q^{\mathrm{Cr}(\pi)} \int_\pi f_1 \otimes f_2 \otimes f_3 \otimes \ldots \otimes f_r\\
& &+\sum_{p=1}^{\min(n_1,n_2)} \sum_{\substack{\pi \in C_2(n_1\otimes n_2 \otimes n_3 \otimes \ldots \otimes n_r)\\ p_\pi=p}} q^{\mathrm{Cr}(\pi)} \int_\pi f_1 \otimes f_2 \otimes f_3 \otimes \ldots \otimes f_r\\
&=& \sum_{\pi \in C_2(n_1\otimes n_2 \otimes n_3 \otimes \ldots \otimes n_r)} q^{\mathrm{Cr}(\pi)} \int_\pi f_1 \otimes f_2 \otimes f_3 \otimes \ldots \otimes f_r,
\eean
which completes the induction procedure.
\end{proof}

\section{Fourth moment theorem and applications}\label{section:proof}

\subsection{Colored pairings}

Given $n_1,\ldots,n_R\geq 1$, we denote by $\mathrm{Col}(\{n_1,\ldots,n_R\})$ the set of `colors' in $\{n_1,\ldots,n_R\}$, i.e., $\mathrm{Col}(\{n_1,\ldots,n_R\})=\{N_1,\ldots,N_K\}$ where $N_1<\ldots<N_K$ are such that $\{N_1,\ldots,N_k\}=\{n_1,\ldots,n_R\}$. We also introduce the set $C_2^{col}(n_{i_1},\ldots,n_{i_r})$ of pairings in $\mathcal{P}_2(\{1,\ldots,r\})$ that respects the coloring of $(n_1,\ldots,n_R)$, i.e.,
$$C_2^{col}(n_{1},\ldots,n_{R}):=\{\pi \in \mathcal{P}_2(\{1,\ldots,r\}): \ \text{for all} \ \{l,m\}\in \pi, \ n_{l}=n_{m}\}.$$
Given $\pi\in C_2^{col}(n_{1} ,  \ldots , n_{R})$, we define $\pi_i$ ($1\leq i\leq K$) as the set of pairs in $\pi$ with the same color $N_i$, i.e.,
$$\pi_i:=\{\{l,m\}\in \pi: \ n_l=n_m=N_i\}.$$

\smallskip

\noindent
Finally, for all subsets $\pi',\pi''$ of a given pairing $\pi$, the notation $\mathrm{Cr}(\pi',\pi'')$ will refer to the number of crossings (in $\pi$) between pairs of $\pi'$ and pairs of $\pi''$.

\

With this notation in hand, we can state the main result of this section.

\begin{theorem}\label{theo:main}
Fix $q\in [0,1)$, $d\geq 1$, $n_1,\ldots,n_d\geq 1$ and consider $d$ sequences $\{f_k^1\}_{k\geq 1}\subset L^2(\R_+^{n_1}),\ldots, \{f_k^d\}_{k\geq 1}\subset  L^2(\R_+^{n_d})$ of symmetric functions. Assume that
\[
\lim_{k\to\infty}
\langle f^i_k,f^j_k\rangle_q = c(i,j)\quad\mbox{for all $i,j\in\{1,\ldots,d\}$}.
\]
Then, the following four assertions are equivalent as $k\to\infty$:

\

\noindent
(i) For every $i\in \{1,\ldots,d\}$, $\vp_q\big( I_{n_i}^{X^{(q)}}(f_k^i)^4 \big) \to (2+q^{n_i^2}) c(i,i)^2$.

\

\noindent
(ii) For every $i\in \{1,\ldots,d\}$ and every $p\in \{1,\ldots,n_i-1\}$, $\|f^i_k \contra{p} f^i_k\|_{L^2(\R_+^{2{n_i}-2p})} \to 0$.

\

\noindent
(iii) For every $i\in \{1,\ldots,d\}$, $I_{n_i}^{X^{(q)}}(f_k^i) \overset{\text{law}}{\to} \mathcal{G}_{q^{n_i^2}}(0,c(i,i))$.

\

\noindent
(iv) For every $r\geq 1$ and every $i_1,\ldots,i_r \in \{1,\ldots,d\}$, one has
\begin{equation}\label{formu-gene}
\vp_q\big( I_{n_{i_1}}^{X^{(q)}}(f_k^{i_1})\ldots I_{n_{i_r}}^{X^{(q)}}(f_k^{i_r}) \big) \to \sum_{\pi \in C_2^{col}(n_{i_1},\ldots,n_{i_r})} \prod_{1\leq i\leq j\leq K} q^{\mathrm{Cr}(\pi_i,\pi_j)N_iN_j} \prod_{\{l,m\}\in \pi} c(i_l,i_m),
\end{equation}
where we have set $\{N_1,\ldots,N_K\}:=\mathrm{Col}(\{n_{i_1},\ldots,n_{i_r}\})$.
\end{theorem}

\

Remark that, in the case where the integers $n_1,\ldots,n_d$ are such that
\[
\text{Card}(\mathrm{Col}(\{n_{1},\ldots,n_{d}\}))\geq 2,
\]
it seems difficult to identify the limit in (\ref{formu-gene}) as the distribution of a multidimensional $q^\ast$-Gaussian law
for some $q^\ast \in (-1,1)$, due to the (possibly) changing weights $q^{\mathrm{Cr}(\pi_i,\pi_j)N_i,N_j}$.

\

Before proving Theorem \ref{theo:main}, let us first explain how it implies Theorem \ref{theo:intro}, Theorem \ref{transferprinciple} and
Theorem \ref{breuermajor}.

\

{\it Proof of Theorem \ref{theo:intro}}. Theorem \ref{theo:intro} is an immediate spin-off of Theorem \ref{theo:main}.
Indeed, take $n_1=\ldots=n_d=n$, $f^1_k=\ldots=f^d_k=f_k$ with $\|f_k\|_q^2 \to 1$ as $k\to\infty$. Then $\mathrm{Col}(\{n_{i_1},\ldots,n_{i_r}\})=\{n\}$, $C_2^{col}(n_{i_1},\ldots,n_{i_r})=\mathcal{P}_2(\{1,\ldots,r\})$, and the limit in (\ref{formu-gene}) becomes $\sum_{\pi \in \mathcal{P}_2(\{1,\ldots,r\})} q^{\mathrm{Cr}(\pi)n^2}$, which is precisely the $r$th moment of the $q$-Gaussian law $\mathcal{G}_{q^{n^2}}(0,1)$, see indeed (\ref{mom-qbm-intro})
and observe that $\mathcal{P}_2(\{1,\ldots,r\})=\emptyset$ when $r$ is odd.\qed

\bigskip

{\it Proof of Theorem \ref{transferprinciple}}. Here again, it is an immediate spin-off of Theorem \ref{theo:main},
still by considering the case where $n_1=\ldots=n_d=n$ and $f^1_k=\ldots=f^d_k=f_k$ with $\|f_k\|_q^2 \to 1$ as $k\to\infty$. Indeed, assume that $I_n^{X^{(q)}}(f_k)$ converges in law to $\mathcal{G}_{q^{n^2}}(0,\si_{q}^2)$
for {\it one particular} $q\in[0,1]$.
Using the implication $(iii)\Rightarrow(ii)$ in Theorem \ref{theo:main}, we get
that $\|f_k \contra{p} f_k\|_{L^2(\R_+^{2n-2p})} \to 0$ as $k\to\infty$ for
every $p\in \{1,\ldots,n-1\}$. But, since this latter assertion does not depend on $q$, we deduce, by using
the converse implication $(ii)\Rightarrow(iii)$ for all the other possible values of $q$, that
$I_n^{X^{(q)}}(f_k)$ converges in law to $\mathcal{G}_{q^{n^2}}(0,\si_{q}^2)$
for {\it all} $q\in[0,1]$.\qed

\bigskip

{\it Proof of Theorem \ref{breuermajor}}.
The result for $q=1$ being already known (see \cite{breuermajor}), we fix $q\in[0,1)$ in all the proof.
Since our aim is to show a convergence in law,  in (\ref{thm19expr}) it is not a loss of generality to
 replace
 $G_l$ with $I_1^{X^{(q)}}(e_l)$, where the sequence $\{e_l\}_{l\in\N}\subset L^2(\R_+)$
is chosen so that
\begin{equation}\label{el}
\int_0^\infty e_l(x)e_m(x)dx = \rho(l-m),\quad l,m\in\N.
\end{equation}
Indeed, using the fact that the covariance function characterizes the distribution of any centered $q$-Gaussian family, we  immediately check that
$\{G_l\}_{l\in\N} \,\overset{\rm law}{=}\,\{ I_1^{X^{(q)}}(e_l)\}_{l\in\N}$. Moreover, such a sequence $\{e_l\}$ is easily shown to exist by considering the linear span
$\mathcal{H}$ of $\{G_l\}$. It is indeed a real separable Hilbert space and, consequently,
there exists an isometry $\Phi:\mathcal{H}\to L^2(\R_+)$. By setting $e_l = \Phi(G_l)$, we get that
 (\ref{el}) holds true.

 Now, using (\ref{lien-tch}), we can write
\[
\frac{1}{\sqrt{k}}\sum_{l=0}^{[kt]}H_n^{(q)}(G_l)
\overset{\rm law}{=}\frac{1}{\sqrt{k}}\sum_{l=0}^{[kt]}H_n^{(q)}\big(I_1^{X^{(q)}}(e_l)\big)
=\frac{1}{\sqrt{k}}\sum_{l=0}^{[kt]}I^{X^{(q)}}_n\big(e_l^{\otimes n}\big)
=I^{X^{(q)}}_n\left(\frac{1}{\sqrt{k}}\sum_{l=0}^{[kt]}e_l^{\otimes n}\right).
\]
Fix $d\geq 1$ and $t_1,\ldots,t_d>0$.
For each $i\in\{1,\ldots,d\}$, observe that the kernel
\[
f_k(t_i):=\frac{1}{\sqrt{k}}\sum_{l=0}^{[kt_i]}e_l^{\otimes n}
\]
is a symmetric
function of $L^2(\R_+^n)$. Moreover, it may be shown (see \cite[chapter 7]{NouPecBook}) that
\[
\lim_{k\to\infty}\langle f_k(t_i),f_k(t_j)\rangle_q =
\lim_{k\to\infty}\sum_{\sigma\in\mathfrak{S}_n}q^{{\rm inv}(\si)}\langle f_k(t_i),f_k(t_j)\rangle_{L^2(\R_+^n)} =
\sum_{\sigma\in\mathfrak{S}_n}q^{{\rm inv}(\si)}\sum_{l\in\Z}\rho(l)^n\,t_i\wedge t_j
\]
for all $i,j\in\{1,\ldots,d\}$
(in particular, $\sum_{l\in\Z}\rho(l)^n\geq 0$) and that
\[
\lim_{k\to\infty}\| f_k(t_i)\contra{p}f_k(t_i)\|_{L^2(\R_+^{2n-2p})} =
0
\]
for all $i\in\{1,\ldots,d\}$ and all $p\in\{1,\ldots,n-1\}$.
Take $n_1=\ldots=n_d=n$ in Theorem \ref{theo:main} and set $f^1_k=f_k(t_1)$, $\ldots$,
$f^d_k=f_k(t_d)$.
Since $(ii)$ holds, we deduce from (\ref{formu-gene}) that, for every $r\geq 1$ and every $s_1,\ldots,s_r\in\{t_1,\ldots,t_d\}$,
the quantity
\[
\vp_q\big( I_{n}^{X^{(q)}}(f_k(s_1))\ldots I_{n}^{X^{(q)}}(f_k(s_r)) \big)
=
\vp\left(\frac{1}{\sqrt{k}}\sum_{l=0}^{[ks_1]}H_n^{(q)}(G_l)\ldots
\frac{1}{\sqrt{k}}\sum_{l=0}^{[ks_r]}H_n^{(q)}(G_l)\right)
\]
converges,
as $k\to\infty$, to
\bean
&&
\left(\sum_{\sigma\in\mathfrak{S}_n}q^{{\rm inv}(\si)}\sum_{l\in\Z}\rho(l)^n\right)^{r/2}\sum_{\pi\in \mathcal{P}_2(\{1,\ldots,r\})} q^{\mathrm{Cr}(\pi)n^2} \prod_{\{l,m\}\in \pi} s_{i_l}\wedge s_{i_m}\\
&&= \left(\sum_{\sigma\in\mathfrak{S}_n}q^{{\rm inv}(\si)}\sum_{l\in\Z}\rho(l)^n\right)^{r/2}\vp_{q_n}\big( X^{(q_n)}(s_1) \ldots X^{(q_n)}(s_r) \big),
\eean
where we set $q_n:=q^{n^2}$ for simplicity and where $X^{(q_n)}$ is a $q_n$-Brownian motion
(for the second expression of the limit, see (\ref{form-q-bm})).
This concludes the proof of Theorem \ref{breuermajor}.\qed

\bigskip

We now turn to the proof of Theorem \ref{theo:main}.

\

{\it Proof of Theorem \ref{theo:main}}.
We shall prove the following sequence of implications: $(iv)\Rightarrow (iii)\Rightarrow (i)\Rightarrow (ii)\Rightarrow (iv)$.

\

\fbox{
\begin{minipage}{0.13\textwidth}
$(iv) \Rightarrow (iii)$
\end{minipage}
}
Assume $(iv)$, fix $i\in\{1,\ldots,d\}$, $r\geq 1$ and take $i_1=\ldots=i_r=i$ in $(iv)$.
Let $f^1,\ldots,f^d\in L^2(\R_+)$ be such that $\langle f^i,f^j\rangle_{L^2(\R_+)}=c(i,j)$ for all $i,j\in\{1,\ldots,d\}$.
In this case, convergence (\ref{formu-gene}) can be written as
\bean
\vp_{q}\big( I_{n_i}^{X^{(q)}}(f_k^i)^r \big)\to&& \sum_{\pi\in \mathcal{P}_2(\{1,\ldots,r\})} q^{\mathrm{Cr}(\pi)n_i^2} \prod_{\{l,m\}\in \pi} \langle f^i,f^i \rangle_{L^2(\R_+)}\\
&=& \sum_{\pi\in \mathcal{P}_2(\{1,\ldots,r\})} q^{\mathrm{Cr}(\pi)n_i^2} \int_\pi f^i \otimes \ldots \otimes f^i\\
&=& \vp_{q_i}\big( I_1^{X^{(q_i)}}(f^i)^r \big)\quad\mbox{with $q_i=q^{n_i^2}$},
\eean
the last equality being an easy consequence of formula (\ref{form-mom}).
Since the convergence in law is equivalent by its very definition to the convergence of all the moments,
the implication $(iv)\Rightarrow (iii)$ is shown.

\

\fbox{
\begin{minipage}{0.11\textwidth}
$(iii) \Rightarrow (i)$
\end{minipage}
}
Obvious (still because convergence in law reduces to convergence of all the moments in our framework).

\

\fbox{
\begin{minipage}{0.1\textwidth}
$(i) \Rightarrow (ii)$
\end{minipage}
}
We shall make use of the following proposition.

\begin{proposition}\label{calcul4}
Fix $q\in (-1,1)$. Then, for every symmetric function $f\in L^2(\R_+^n)$, we have
\begin{equation}\label{mom-4}
\vp_q\big( I_n^{X^{(q)}}(f)^4 \big)=(2+q^{n^2}) \|f\|_q^4+\sum_{p=1}^{n-1} \bigg\{ \| f\pt{p} f\|_{q}^2+\bigg( \sum_{\si\in \mathfrak{S}_{2n}^p}q^{\mathrm{inv}(\si)}\bigg)  \|f \contra{p}f\|^2_{L^2(\R_+^{2n-2p})}\bigg\},
\end{equation}
where
$$ \mathfrak{S}_{2n}^p=\{\si \in \mathfrak{S}_{2n}: \ \mathrm{Card}(\si(\{1,\ldots,n\})\cap \{1,\ldots,n\})=p\}.$$
\end{proposition}

Before proving it, let us see how it implies the desired conclusion.
Fix $i\in \{1,\ldots,d\}$ and assume that the sequence $\{f^i_k\}$ of Theorem \ref{theo:main} satisfies $(i)$, that is,
$$\lim_{k\to\infty}\vp_q\big( I_{n_i}^{X^{(q)}}(f_k^i)^4 \big) = (2+q^{n_i^2}) c(i,i)^2=\lim_{k\to \infty} \ (2+q^{n_i^2}) \|f_k^i\|_q^2.$$
Thanks to (\ref{mom-4}), we deduce that
\begin{equation}\label{trans}
\sum_{p=1}^{n_i-1}\left\{\| f^i_k\pt{p} f^i_k\|_{q}^2+\bigg( \sum_{\si\in \mathfrak{S}_{2n}^p}q^{\mathrm{inv}(\si)}\bigg)  \|f^i_k \contra{p}f^i_k\|^2_{L^2(\R_+^{2n_i-2p})}\right\}\to 0\quad\mbox{as $k\to\infty$}.
\end{equation}
If $q=0$, we have $\|f_k^i \pt{p}f_k^i\|_q=\|f^i_k \contra{p}f^i_k\|^2_{L^2(\R_+^{2n_i-2p})}$
and $ \sum_{\si\in \mathfrak{S}_{2n}^p}q^{\mathrm{inv}(\si)}=0$; we deduce that
$\|f^i_k \contra{p}f^i_k\|^2_{L^2(\R_+^{2n_i-2p})} \to 0$ for every $p\in \{1,\ldots,n_i-1\}$, which is precisely $(ii)$. If $q>0$, then $\sum_{\si\in \mathfrak{S}_{2n}^p}q^{\mathrm{inv}(\si)} >0$ and, accordingly, we also get that $\|f^i_k \contra{p}f^i_k\|^2_{L^2(\R_+^{2n_i-2p})} \to 0$ for every $p\in \{1,\ldots,n_i-1\}$, that is, $(ii)$ holds true as well.

\

{\it Proof of Proposition \ref{calcul4}}.
By using the product formula (\ref{mul}) together with the isometry (\ref{isometrie}), we first deduce
$$\vp_q\big( I_n^{X^{(q)}}(f)^4\big) =\vp_q\left[\bigg( \sum_{p=0}^n I_{2n-2p}(f \pt{p} f) \bigg)^2\right]=\sum_{p=0}^n \|f\pt{p}f\|_q^2.$$
Next, observe that $(f\pt{n} f)^2=\|f\|_q^4$ and, using (\ref{prod-scal-2}), write $\|f\otimes f\|_q^2$ as
$$\|f\otimes f\|_q^2=\sum_{\si\in \mathfrak{S}_{2n}} q^{\mathrm{inv}(\si)}\int_{P_2(\si)} (f\otimes f) \otimes (f\otimes f)^\ast=\sum_{p=0}^n\sum_{\si\in \mathfrak{S}_{2n}^p} q^{\mathrm{inv}(\si)}\int_{P_2(\si)} (f\otimes f) \otimes (f\otimes f)^\ast.$$
It is readily checked that
$$\sum_{\si\in \mathfrak{S}_{2n}^n} q^{\mathrm{inv}(\si)}\int_{P_2(\si)} (f\otimes f) \otimes (f\otimes f)^\ast=\bigg( \sum_{\si \in \mathfrak{S}_n} q^{\mathrm{inv}(\si)} \int_{P_2(\si)} f\otimes f \bigg)^2=\|f\|_q^4,$$
while
$$\sum_{\si\in \mathfrak{S}_{2n}^0} q^{\mathrm{inv}(\si)}\int_{P_2(\si)} (f\otimes f) \otimes (f\otimes f)^\ast=q^{n^2}\bigg( \sum_{\si \in \mathfrak{S}_n} q^{\mathrm{inv}(\si)} \int_{P_2(\si)} f\otimes f \bigg)^2=q^{n^2}\|f\|_q^4.$$
Also, since $f$ is a symmetric function, it is easy to verify that for every $\si\in \mathfrak{S}_{2n}^p$ with $p\in \{1,\ldots,n-1\}$, one has
$$\int_{P_2(\si)} (f\otimes f) \otimes (f\otimes f)^\ast=\| f \contra{p} f\|_{L^2(\R_+^{2n-2p})}^2.$$
As a result, we deduce that
$$\|f\otimes f\|_q^2=(1+q^{n^2})\|f\|_q^4+\sum_{p=1}^{n-1}\big( \sum_{\si\in \mathfrak{S}_{2n}^p}q^{\mathrm{inv}(\si)}\big) \cdot \|f \contra{p}f\|_{L^2(\R_+^{2n-2p})}^2.$$
\qed

\begin{remark}\label{rk:2}
Assume that $n=2$ and let $f\in L^2(\R_+^2)$ be a symmetric function.
It is readily checked that
$\|f\pt{1}f\|^2_q = (1+q)^5 \|f\contra{1}f\|^2_{L^2(\R_+^2)}$
We deduce from (\ref{mom-4}) that
\begin{equation}\label{zer}
\vp_q\big( I_2^{X^{(q)}}(f)^4 \big)=(2+q^{4}) \|f\|_q^4+(1+q)^4(2q+1)  \|f \contra{1}f\|^2_{L^2(\R_+^{2})}.
\end{equation}
In particular, when $q=-1/2$ and $f=\sqrt{2}\,{\bf 1}_{[0,1]^2}$, we have
$\varphi_q(I_2^{X^{(q)}}(f)^2)=\|f\|_q^2=1$ and
$\varphi_q(I_2^{X^{(q)}}(f)^4)=2+q^4$. But $I_2^{X^{(q)}}(f)$ is not $\mathcal{G}_{q^4}(0,1)$-distributed since,
by using the product formula (\ref{mul}) together with the isometry (\ref{isometrie}) and the fact that
$f\pt{1}f=(1+q)^2\,f\contra{1}f=\sqrt{2}(1+q)^2\,f$, we have
\begin{eqnarray*}
\varphi_q(I_2^{X^{(q)}}(f)^3)
&=&\varphi_q\left(I_2^{X^{(q)}}(f)\, \big(I_{4}^{X^{(q)}}(f \otimes f) + I_{2}^{X^{(q)}}(f \pt{1} f) + \|f\|_q^2 \big)\right)\\
&=&\langle f,f \pt{1} f\rangle_q
=\sqrt{2}(1+q)^2\neq 0.
\end{eqnarray*}
This explicit situation shows that Theorem \ref{theo:main} may actually be false when $q$ is negative.
However, formula (\ref{zer}) implies the following interesting fact when $n=2$. Let $q\in(-1,-\frac12)\cup(-\frac12,1]$ and let $\{f_k\}_{k\geq 1}$ be a sequence of symmetric functions in $L^2(\R_+^2)$ satisfying
$\|f_k\|_{q} \to 1$ as $k\to\infty$.
Then, due to (\ref{zer}) we have equivalence between $(a)$ $\vp_q\big( I_2^{X^{(q)}}(f_k)^4 \big) \to 2+q^{4}$
and $(b)$ $\|f_k \contra{1}f_k\|^2_{L^2(\R_+^{2})}\to 0$ as $k\to\infty$.
Following the same line of reasoning as in the forthcoming proof of $(ii)\Rightarrow (iv)$, we may deduce that
$(a)-(b)$ are also equivalent to $(c)$ $I_2^{X^{(q)}}(f_k)\overset{\rm law}{\to} \mathcal{G}_{q^4}(0,1)$. That is, when $n=2$,
Theorem \ref{theo:main} happens to hold true for all $q\in(-1,-\frac12)\cup(-\frac12,1]$.\qed
\end{remark}

\

\fbox{
\begin{minipage}{0.12\textwidth}
$(ii) \Rightarrow (iv)$
\end{minipage}
}
$\mbox{ }$\,
Assume $(ii)$ and let us show $(iv)$.
Our starting point is the formula (\ref{form-mom}), which yields
\begin{equation}\label{start}
\vp_q\big( I_{n_{i_1}}^{X^{(q)}}(f_k^{i_1}) \ldots I_{n_{i_{r}}}^{X^{(q)}}(f^{i_{r}}_k)\big)=\sum_{\pi\in C_2(n_{i_1}\otimes \ldots \otimes n_{i_r})} q^{\mathrm{Cr}(\pi)} \int_\pi f^{i_1}_k \otimes \ldots \otimes f^{i_r}_k.
\end{equation}
Let us introduce the subset $C^0_2(n_{i_1}\otimes \ldots \otimes n_{i_r})$ of pairings $\pi\in C_2(n_{i_1}\otimes \ldots \otimes n_{i_r})$ that meet the  following two conditions:

\smallskip

\noindent
(A): If two blocks of $n_{i_1}\otimes \ldots \otimes n_{i_r}$ are linked, then they necessarily have the same cardinality.

\smallskip

\noindent
(B): Any two blocks of $n_{i_1}\otimes \ldots \otimes n_{i_r}$ having the same number $n$ of elements are necessarily linked by $0$ or $n$ pair(s).

\

Let $\pi \in C_2(n_{i_1}\otimes \ldots \otimes n_{i_r})\backslash C^0_2(n_{i_1}\otimes \ldots \otimes n_{i_r})$ and assume that $\pi$ does not satisfy (A), i.e., that there exists two blocks (say the $l$th and the $m$th blocks) with different cardinalities ($n_{i_l}<n_{i_m}$) that are connected in $\pi$ by $p$ pairs, for some $p\in \{1,\ldots,n_{i_l}\}$. Then there exists $\pi'\in C_2((n_{i_l}+n_{i_m}-2p) \otimes n_{i_1} \otimes \ldots \otimes \widehat{n_{i_l}}\otimes \ldots \otimes \widehat{n_{i_m}} \otimes \ldots \otimes n_{i_r})$ such that
$$\int_\pi f^{i_1}_k \otimes \ldots \otimes f^{i_r}_k=\int_{\pi'} (f^{i_l}_k \contra{p} f_k^{i_m}) \otimes f_k^{i_1} \otimes \ldots \otimes \widehat{f^{i_l}_k} \otimes \ldots \otimes \widehat{f^{i_m}_k}\otimes \ldots \otimes f^{i_{r}} ,$$
where the notation \ $\widehat{}$ \ means that this particular object is removed. By using Cauchy-Schwarz, we  get
$$
\big| \int_\pi f^{i_1}_k \otimes \ldots \otimes f^{i_{r}}_k\big| \leq \|f^{i_l}_k \contra{p} f_k^{i_m}\|_{L^2(\R_+^{n_{i_l}+n_{i_m}-2p})} \prod_{s\in \{1,\ldots,r\} \backslash \{l,m\}} \|f^{i_s}_k\|_{L^2(\R_+^{n_{i_s}})}.
$$
Now, some elementary computations yield
\bean
\|f^{i_l}_k \contra{p} f_k^{i_m}\|_{L^2(\R_+^{n_{i_l}+n_{i_m}-2p})}^2 &=&\langle f_k^{i_l} \contra{n_{i_l}-p} f_k^{i_l},f_k^{i_m} \contra{n_{i_m}-p} f_k^{i_m} \rangle_{L^2(\R_+^{2p})}\\
&\leq &\|f_k^{i_l} \contra{n_{i_l}-p} f_k^{i_l} \|_{L^2(\R_+^{2p})} \|f_k^{i_m} \contra{n_{i_m}-p} f_k^{i_m}\|_{L^2(\R_+^{2p})}.
\eean
Since $n_{i_m}-p\in \{1,\ldots,n_{i_m}-1\}$, we know that $\|f_k^{i_m} \contra{n_{i_m}-p} f_k^{i_m}\|_{L^2(\R_+^{2p})} \to 0$ as $k\to\infty$, and accordingly $\int_\pi f^{i_1}_k \otimes \ldots \otimes f^{i_r}_k \to 0$.

\smallskip

\noindent
On the other hand, let $\pi\in C_2(n_{i_1}\otimes \ldots \otimes n_{i_r})\backslash C^0_2(n_{i_1}\otimes \ldots \otimes n_{i_r})$ be an element not satisfying (B), i.e., such that there exists two blocks (say the $l$th and the $m$th blocks) with the same cardinality $n$ that are connected in $\pi$ by $p$ pairs, for some $p\in \{1,\ldots,n-1\}$. Using similar arguments as above, we get
$$
\big| \int_\pi f^{i_1}_k \otimes \ldots \otimes f^{i_{r}}_k\big|^2 \leq \|f_k^{i_l} \contra{n-p} f_k^{i_l} \|_{L^2(\R_+^{2p})} \|f_k^{i_m} \contra{n-p} f_k^{i_m}\|_{L^2(\R_+^{2p})} \prod_{s\in \{1,\ldots,r\} \backslash \{l,m\}} \|f^{i_s}_k\|_{L^2(\R_+^{n_{i_s}})}^2,
$$
and, since $n-p\in \{1,\ldots,n-1\}$, we can here again conclude that $\int_\pi f^{i_1}_k \otimes \ldots \otimes f^{i_{2r}}_k \to 0$.

\

Going back to (\ref{start}), we deduce that
\begin{equation}\label{start-1}
\vp_q\big( I_{n_{i_1}}^{X^{(q)}}(f_k^{i_1}) \ldots I_{n_{i_{r}}}^{X^{(q)}}(f^{i_{r}}_k)\big)-\sum_{\pi\in C_2^0(n_{i_1}\otimes \ldots \otimes n_{i_r})} q^{\mathrm{Cr}(\pi)} \int_\pi f^{i_1}_k \otimes \ldots \otimes f^{i_r}_k \to 0
\end{equation}
as $k\to\infty$.
Now, with every $\pi\in C_2^0(n_{i_1}\otimes \ldots \otimes n_{i_r})$, we associate a pairing $\widetilde{\pi}\in \mathcal{P}_2(\{1,\ldots,r\})$ according to the basic following procedure: two points $l<m\in \{1,\ldots,r\}$ are linked in $\widetilde{\pi}$ if the $l$th and $m$th blocks of $n_{i_1} \otimes \ldots \otimes n_{i_r}$ are linked in $\pi$ (see Figure 2). Owing to the two conditions (A) and (B), it is clear that $\widetilde{\pi}$ defines an element of $C_2^{col}(n_{i_1},\ldots,n_{i_r})$. Then, divide $\widetilde{\pi}$ into the disjoint subsets $\widetilde{\pi}_1,\ldots,\widetilde{\pi}_K$, where $\widetilde{\pi}_t$ stands for the set of pairs having the same color $N_t$. In the sequel, we will use the following explicit notation
$$\widetilde{\pi}_t=\{\{l_s^{(t)},m_s^{(t)}\}, \ l_s^{(t)} < m_s^{(t)}, \ 1\leq s \leq R_t\} \quad , \quad t\in \{1,\ldots,K\}.$$
Conversely, it is clear that any pairing $\pi \in  C_2^0(n_{i_1}\otimes \ldots \otimes n_{i_r})$ can be reconstructed from the two following
ingredients:

\smallskip

\noindent
$\bullet$ its "projection" $\widetilde{\pi}$,

\smallskip

\noindent
$\bullet$ the description of the links in $\pi$ coded by a single pair in $\widetilde{\pi}$. This description can be made clear through a set of permutations $\si_s^{(t)} \in \mathfrak{S}_{N_t}$ associated with each pair $(l_s^{(t)},m_s^{(t)})$ of $\widetilde{\pi}$, according to the following principle: in $\pi$, the links between the $l_s^{(t)}$th and $m_s^{(t)}$th blocks are given (when considering these blocks isolated from the others) by the set
$$\{(N_t+1-i,N_t+\si_s^{(t)}(i)), \ 1\leq i\leq N_t\}.$$

\smallskip

With this identification in mind, our computation of $\text{Cr}(\pi)$ for $\pi \in C_2^0(n_{i_1}\otimes \ldots \otimes n_{i_r})$ relies on the two following observations:

\smallskip

\noindent
$\bullet$ a crossing in $\widetilde{\pi}$ between a pair in $\widetilde{\pi}_i$ and a pair in $\widetilde{\pi}_j$ gives birth to $N_iN_j$ crossings in $\pi$.

\smallskip

\noindent
$\bullet$ the number of crossings in $\pi$ between the pairs that connect the $l_s^{(t)}$th and $m_s^{(t)}$th blocks is given by $\text{inv}(\si_s^{(t)})$.

\smallskip

From the above considerations we deduce the following formula: for every fixed $\pi \in C_2^0(n_{i_1}\otimes \ldots \otimes n_{i_r})$, one has
$$
q^{\mathrm{Cr}(\pi)} \int_\pi f^{i_1}_k \otimes \ldots \otimes f^{i_r}_k
=\prod_{1\leq i\leq j\leq K} q^{\mathrm{Cr}(\widetilde{\pi}_i,\widetilde{\pi}_j)N_iN_j} \prod_{1\leq t\leq K} \prod_{1\leq s\leq R_t} q^{\mathrm{inv}(\si_s^{(t)})} \int_{P_2(\si_s^{(t)})} f_k^{i_{l_s^{(t)}}} \otimes f_k^{i_{m_s^{(t)}}}.
$$
Therefore,
\begin{multline*}
\sum_{\pi \in C_2^0(n_{i_1}\otimes \ldots \otimes n_{i_r})}q^{\mathrm{Cr}(\pi)} \int_\pi f^{i_1}_k \otimes \ldots \otimes f^{i_r}_k\\
=\sum_{\widetilde{\pi}\in C_2^{col}(n_{i_1},\ldots,n_{i_r})} \prod_{1\leq i\leq j\leq K} q^{\mathrm{Cr}(\widetilde{\pi}_i,\widetilde{\pi}_j)N_iN_j}\prod_{1\leq t\leq K} \prod_{\{l,m\}\in \widetilde{\pi}_t} \sum_{\si\in \mathfrak{S}_{N_t}} q^{\mathrm{inv}(\si)} \int_{P_2(\si)} f_k^{i_l} \otimes f_k^{i_m}.
\end{multline*}
At this point, let us recall that
$$\sum_{\si\in \mathfrak{S}_{N_t}} q^{\mathrm{inv}(\si)} \int_{P_2(\si)} f_k^{i_l} \otimes f_k^{i_m}=\langle f^{i_l}_k,f^{i_m}_k\rangle_q \to c(i_l,i_m)$$
as $k\to\infty$, so that
$$\sum_{\pi \in C_2^0(n_{i_1}\otimes \ldots \otimes n_{i_r})}q^{\mathrm{Cr}(\pi)} \int_\pi f^{i_1}_k \otimes \ldots \otimes f^{i_r}_k\to \sum_{\pi \in C_2^{col}((n_{i_1},\ldots,n_{i_r}))} \prod_{1\leq i\leq j\leq K} q^{\mathrm{Cr}(\pi^i,\pi^j)N_iN_j} \prod_{\{l,m\}\in \pi} c(i_l,i_m)$$
as $k\to\infty$ which, thanks to (\ref{start-1}), entails (\ref{formu-gene}). The proof of Theorem \ref{theo:main}
is complete.
\qed

\begin{center}
\begin{figure}[!ht]\label{pic:pairing}

\begin{pspicture}(0,-5.5)(13.4,8)


\psline[linecolor=blue](0,0)(0,6.6) \psline[linecolor=blue](0,6.6)(11.4,6.6) \psline[linecolor=blue](11.4,6.6)(11.4,0)
\psline[linecolor=blue](0.4,0)(0.4,7) \psline[linecolor=blue](0.4,7)(11.8,7) \psline[linecolor=blue](11.8,7)(11.8,0)

\psline[linecolor=green](1,0)(1,6) \psline[linecolor=green](1,6)(7,6) \psline[linecolor=green](7,6)(7,0)

\psline[linecolor=red](1.6,0)(1.6,5.4) \psline[linecolor=red](1.6,5.4)(6,5.4) \psline[linecolor=red](6,5.4)(6,0)
\psline[linecolor=red](2,0)(2,5) \psline[linecolor=red](2,5)(5.6,5) \psline[linecolor=red](5.6,5)(5.6,0)
\psline[linecolor=red](2.4,0)(2.4,4.6) \psline[linecolor=red](2.4,4.6)(6.4,4.6) \psline[linecolor=red](6.4,4.6)(6.4,0)

\psline[linecolor=blue](3,0)(3,4) \psline[linecolor=blue](3,4)(9.4,4) \psline[linecolor=blue](9.4,4)(9.4,0)
\psline[linecolor=blue](3.4,0)(3.4,3.6) \psline[linecolor=blue](3.4,3.6)(9,3.6) \psline[linecolor=blue](9,3.6)(9,0)

\psline[linecolor=green](4,0)(4,3) \psline[linecolor=green](4,3)(12.4,3) \psline[linecolor=green](12.4,3)(12.4,0)

\psline[linecolor=blue](4.6,0)(4.6,2) \psline[linecolor=blue](4.6,2)(13,2) \psline[linecolor=blue](13,2)(13,0)
\psline[linecolor=blue](5,0)(5,2.4) \psline[linecolor=blue](5,2.4)(13.4,2.4) \psline[linecolor=blue](13.4,2.4)(13.4,0)

\psline[linecolor=red](7.6,0)(7.6,1.4) \psline[linecolor=red](7.6,1.4)(10,1.4) \psline[linecolor=red](10,1.4)(10,0)
\psline[linecolor=red](8,0)(8,1) \psline[linecolor=red](8,1)(10.4,1) \psline[linecolor=red](10.4,1)(10.4,0)
\psline[linecolor=red](8.4,0)(8.4,0.6) \psline[linecolor=red](8.4,0.6)(10.8,0.6) \psline[linecolor=red](10.8,0.6)(10.8,0)

\rput(0,0){$\bullet$}\rput(11.8,0){$\bullet$} \rput(7,0){$\bullet$} \rput(6,0){$\bullet$} \rput(5.6,0){$\bullet$}
\rput(6.4,0){$\bullet$} \rput(9.4,0){$\bullet$} \rput(9,0){$\bullet$} \rput(12.4,0){$\bullet$} \rput(13,0){$\bullet$}
\rput(13.4,0){$\bullet$} \rput(10,0){$\bullet$} \rput(10.4,0){$\bullet$} \rput(10.8,0){$\bullet$} \rput(11.4,0){$\bullet$} \rput(7.6,0){$\bullet$}

\rput(0,0){$\bullet$} \rput(0.4,0){$\bullet$} \rput(1,0){$\bullet$} \rput(1.6,0){$\bullet$} \rput(2,0){$\bullet$}
\rput(2.4,0){$\bullet$} \rput(3,0){$\bullet$} \rput(3.4,0){$\bullet$} \rput(4,0){$\bullet$}  \rput(4.6,0){$\bullet$}
\rput(5,0){$\bullet$} \rput(7,0){$\bullet$} \rput(8,0){$\bullet$} \rput(8.4,0){$\bullet$}

\psline[linecolor=blue](0.2,-5)(0.2,-2) \psline[linecolor=blue](0.2,-2)(11.6,-2) \psline[linecolor=blue](11.6,-2)(11.6,-5)

\psline[linecolor=green](1,-5)(1,-2.4) \psline[linecolor=green](1,-2.4)(7,-2.4) \psline[linecolor=green](7,-2.4)(7,-5)

\psline[linecolor=red](2,-5)(2,-2.8) \psline[linecolor=red](2,-2.8)(6,-2.8) \psline[linecolor=red](6,-2.8)(6,-5)

\psline[linecolor=blue](3.2,-5)(3.2,-3.2) \psline[linecolor=blue](3.2,-3.2)(9.2,-3.2) \psline[linecolor=blue](9.2,-3.2)(9.2,-5)

\psline[linecolor=green](4,-5)(4,-3.6) \psline[linecolor=green](4,-3.6)(12.4,-3.6) \psline[linecolor=green](12.4,-3.6)(12.4,-5)

\psline[linecolor=blue](4.8,-5)(4.8,-4) \psline[linecolor=blue](4.8,-4)(13.2,-4) \psline[linecolor=blue](13.2,-4)(13.2,-5)

\psline[linecolor=red](8,-5)(8,-4.4) \psline[linecolor=red](8,-4.4)(10.4,-4.4) \psline[linecolor=red](10.4,-4.4)(10.4,-5)

\rput(0.2,-5){$\bullet$} \rput(1,-5){$\bullet$} \rput(2,-5){$\bullet$} \rput(3.2,-5){$\bullet$} \rput(4,-5){$\bullet$}
\rput(4.8,-5){$\bullet$} \rput(6,-5){$\bullet$}  \rput(7,-5){$\bullet$} \rput(8,-5){$\bullet$} \rput(9.2,-5){$\bullet$}
\rput(10.4,-5){$\bullet$} \rput(11.6,-5){$\bullet$} \rput(12.4,-5){$\bullet$} \rput(13.2,-5){$\bullet$}

\end{pspicture}

\

\caption{A pairing $\pi$ in $C_2^0(2\otimes 1\otimes 3\otimes 2\otimes 1\otimes 2\otimes 3\otimes 1 \otimes 3\otimes 2\otimes 3\otimes 2\otimes 1\otimes 2)$ (above) projected in $\widetilde{\pi}$ (below). Here, $\widetilde{\pi}_1=\{(2,8),(5,13)\}$, $\widetilde{\pi}_2=\{(1,12),(4,10),(6,14)\}$, $\widetilde{\pi}_3=\{(3,7),(9,11)\}$.}
\end{figure}
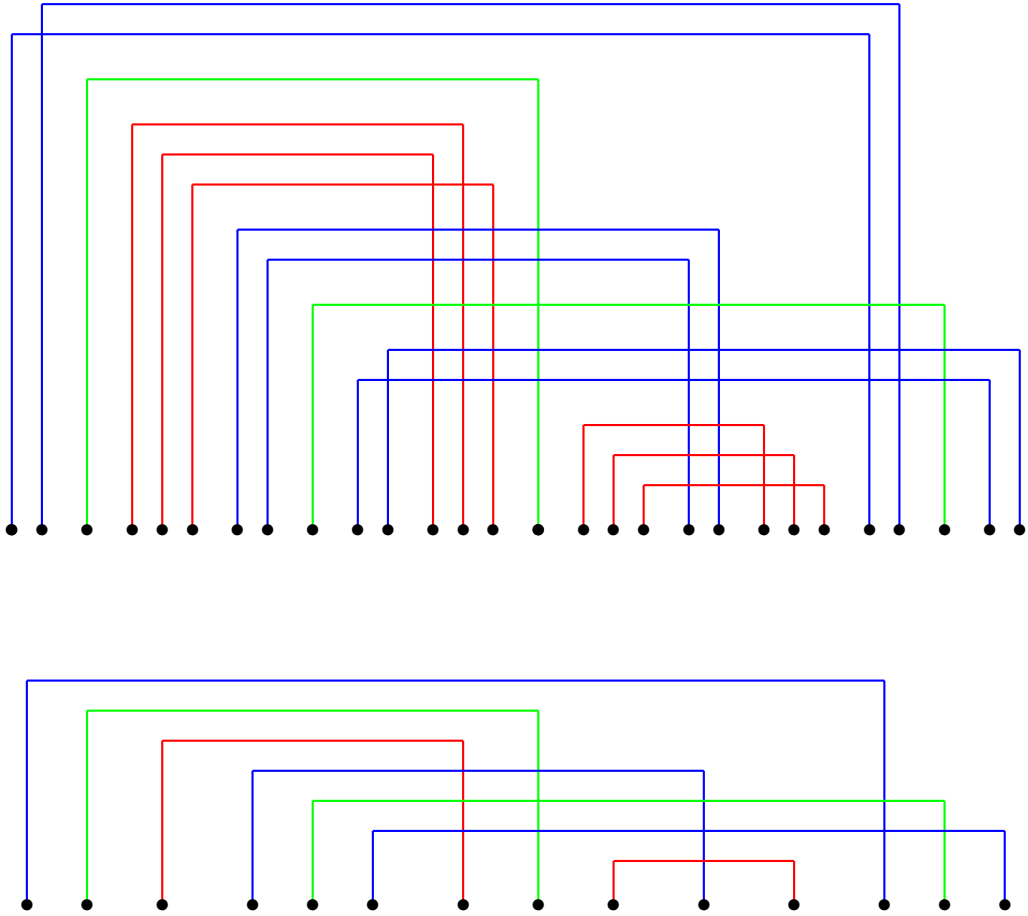
\end{center}

\end{document}